\documentclass[12pt, english, a4paper]{article}
\usepackage[latin9]{inputenc}
\usepackage[T1]{fontenc}
\usepackage{graphicx}
\usepackage{amsthm}
\usepackage{amsfonts}
\usepackage{mathtools}
    \allowdisplaybreaks
\usepackage{amssymb}
\usepackage{makecell}
\usepackage{comment}
\usepackage{adjustbox}
\usepackage{cases}
\usepackage{hyperref}
\usepackage{mathrsfs}
\usepackage{fullpage}
\usepackage{times}
\usepackage[usenames]{color}
\usepackage[dvipsnames]{xcolor}
\usepackage{hyperref}
\usepackage{tikz,tikz-cd}
\usetikzlibrary{decorations.pathreplacing,calligraphy}
\usepackage{extarrows}
\usepackage{float}
\usepackage{ytableau}
\usepackage{xcolor}
\usepackage{circledsteps}
 \usepackage{multirow}
 \usepackage{lmodern}
 
\definecolor{lgrey}{rgb}{0.8,0.8,0.8 }
\definecolor{dgrey}{rgb}{0.3,0.3,0.3 }

\newcommand{\SSS}{\mathfrak{S}}

\theoremstyle{plain}
\newtheorem{theorem}{Theorem}
\newtheorem{note}{Note}

\newtheorem{lemma}[theorem]{Lemma}

\theoremstyle{definition}
\newtheorem{definition}[theorem]{Definition}

\theoremstyle{remark}

\usepackage{authblk}

\title{Matchings and shape-Wilf-Equivalence of sets of patterns of length three II: Quadruples and Quintuples }

\author[1]{Sucharita Biswas \thanks{\tt{biswas.sucharita56@gmail.com}}}
\author[2]{Umesh Shankar\thanks{\tt{ umeshshankar@outlook.com}}} 
\author[3]{Sivaramakrishnan Sivasubramanian \thanks{\tt{ krishnan@math.iitb.ac.in }}}
\affil[1,3]{Department of Mathematics, Indian Institute of Technology, Bombay Mumbai 400076, India} 
\affil[2]{Department of Computer Science and Automation, Indian Institute of Science, Bengaluru 560012, India}
\date{}
\begin{document}

\maketitle

\begin{abstract}
Building on our classification of shape-Wilf-equivalence classes for triples of patterns of length three, we complete the classification for quadruples and quintuples. The larger pattern sets exhibit structural features that are not captured by the encoding methods used for triples and require additional combinatorial tools. Our main new ingredient is a Dyck-path approach to Ferrers boards containing the staircase board. By introducing distinguished cells and related structural parameters on the associated Dyck paths, we obtain new shape-Wilf-equivalences and, for several classes, explicit formulas for the number of avoiding transversals on a fixed Ferrers board.
We further study the corresponding pattern-avoiding perfect matchings through the pattern-preserving bijection of Bloom and Elizalde. Besides deriving recurrence relations, we give bijective interpretations of several matching enumeration sequences in terms of grand Dyck paths, Schr\"oder paths, and Dyck paths. These constructions provide additional combinatorial structure beyond the classification itself.
Together with the companion paper on triples, the present work completes the classification of shape-Wilf-equivalence classes for subsets of patterns of length three, apart from the trivial empty and full pattern sets.
\end{abstract}

\textbf{\small{Keyword:}}{ permutation patterns, Ferrers boards, pattern-avoiding transversals, pattern-avoiding matchings. Dyck paths, Schr\"oder paths, lattice paths.}{\let\thefootnote\relax}

\section{Introduction}
Pattern avoidance on Ferrers boards provides a natural 
extension of the classical theory of permutation 
patterns. In this setting, the appropriate analogue of 
Wilf-equivalence is \textit{shape-Wilf-equivalence}, 
introduced by Backelin, West, and Xin in \cite{bwx-main}. 
Let $F=(F_1,\dots,F_n)$ be a Ferrers board. A \textit{transversal} of $F$ is a placement of $1$'s and $0$'s in the cells of $F$ such that each row and each column contains precisely one $1$. The collection of all transversals of $F$ is denoted by $\SSS_F$. Ordinary permutations can be viewed as special cases of such transversals: a permutation $\pi=\pi_1\pi_2\cdots\pi_n\in \SSS_n$ determines a transversal of the $n\times n$ square Ferrers board by placing a $1$ in the cell $(i,\pi_i)$ for every $i$, leaving all remaining cells equal to $0$.
Let $\SSS_F(P)$ denote the set of transversals of $F$ avoiding every pattern in a set $P$. Two sets of patterns $P$ and $Q$ are said to be shape-Wilf-equivalent if $$|\SSS_F(P)|=|\SSS_F(Q)|$$ for every Ferrers board $F$ and we denote it by $P \sim_s Q$. 
Note that a Ferrers board with $n$ rows and $n$ 
columns admits a transversal only if it contains the 
staircase board. 
Boards not containing the staircase board will only 
contribute zero transversals, and hence it suffices 
to consider boards that contain the staircase.

The study of shape-Wilf-equivalence has deep connections with several combinatorial structures, including matchings, set partitions, and Dyck paths \cite{matching_partition, partition-three, dyckpath-pattern, matchings-partial, pattern-partition}. Among the foundational results in this area is the theorem of Backelin, West and Xin \cite{bwx-main} showing that the monotone increasing and decreasing classical patterns of the same length are shape-Wilf-equivalent. Additional developments by Krattenthaler \cite{growth-diagram} and Bloom--Elizalde \cite{matching_partition} further strengthened the connections between shape-Wilf-equivalence and the enumeration of restricted matchings and set partitions.

The classification problem for patterns of length $3$ (called short patterns) has attracted considerable attention. The results in Stankova and West \cite{new-wilf}, together with those of Backelin, West and Xin \cite{bwx-main} complete the classification for single patterns of length 3, while Bloom and Elizalde \cite{matching_partition} classified the shape-Wilf-equivalence classes for pairs of patterns of length $3$.  The monograph by Kitaev \cite{kitaev-patterns-words}
has a wealth of information on  several variants of pattern avoidance.

In the first part of this work \cite{SUK-arxiv}, we completely determined the shape-Wilf-equivalence classes for triples of patterns of length $3$. 
%The present paper continues that investigation by considering larger collections of patterns. 
In this work, we provide a complete classification of all shape-Wilf-equivalence 
classes arising from 
quadruples and quintuples of patterns of length $3$. Some definitions not
found in this work can be found in \cite{SUK-arxiv}.  Our results are 
summarized in Table \ref{table2} for quadruples and in Table \ref{table5}
for quintuples.  This is covered in Sections \ref{SWE-4} and \ref{SWE-5} 
respectively.
Although several techniques from the study of triples remain applicable, larger pattern sets exhibit additional structural phenomena that require further refinement of the underlying arguments.

%As done in the case of triples, 
The classification of shape-Wilf-equivalence classes for quadruples cannot be
achieved solely through the encoding techniques that we developed earlier for 
triples. Consequently, additional combinatorial tools are required. In particular, 
when a Ferrers board contains a staircase shape, its boundary naturally determines a 
Dyck path. Exploiting this correspondence, we introduce certain distinguished cells 
associated with Dyck paths. These cells enable us to not only establish several 
shape-Wilf-equivalences, but also to explicitly enumerate transversals 
avoiding a given pattern set in each Ferrers board.

A fundamental connection between Ferrers board transversals and matchings was explored by Bloom and Elizalde \cite{matching_partition}, who constructed a bijection between transversals avoiding a pattern set $P$ and matchings avoiding the same set of patterns. Therefore, enumerating $P$-avoiding matchings yields the number of $P$-avoiding transversals of a Ferrers board $F$.
% In Part I\red{(link)}, the corresponding enumeration problems were studied through recursive constructions together with bijections to ternary trees. In contrast, the present paper develops new enumerative interpretations. In Section \ref{matching}, besides deriving several recurrence relations, we establish bijections between restricted matchings and certain classes of lattice paths. These bijective descriptions appear to be new and provide a different combinatorial perspective on the structure of shape-Wilf-equivalence classes for larger pattern sets.

\subsection{Outline of this paper and proof strategies}

The organization of this paper closely parallels that of Part I \cite{SUK-arxiv}. In Sections \ref{SWE-4} and \ref{SWE-5}, we classify the shape-Wilf-equivalence classes of quadruple and quintuple pattern sets of length $3$, respectively.

Several of these equivalences are proved using the encoding method developed in Part I \cite{SUK-arxiv}. We briefly recall the underlying idea. Fix a Ferrers board $F$ and a $P$-avoiding transversal $T$. The $1$'s of $T$ are removed and then placed back row by row from top to bottom. After each placement, the corresponding row and column are removed from further consideration. At every stage, one determines the admissible positions for the next $1$, and the chosen position is recorded by a letter of an encoding word. Conversely, given such a word, the transversal is reconstructed recursively by placing the $1$'s according to the corresponding admissible choices. Thus, when two pattern sets admit the same collection of encoding words and the decoding is unique for both, the common encoding yields a bijection between their avoiding transversals.

In the present paper, the equivalences established in Sections \ref{quad-27811}, \ref{quad-612}, \ref{quad-410}, and \ref{SWE-5} are obtained by variants of this method. Since the general encoding and decoding procedures are identical to those used in Part I \cite{SUK-arxiv}, we do not repeat all routine details. Instead, we state the structural lemmas needed to determine the admissible and forced positions and to verify that the corresponding encodings are compatible.

For some quadruple classes, however, the encoding method alone is not sufficient. To handle these cases, we use a different approach based on Dyck paths associated with Ferrers boards containing the staircase board. The boundary of such a Ferrers board naturally determines a Dyck path. In Sections \ref{quad-1215} and \ref{quad-5614}, we introduce distinguished cells and a notion of distance between peaks, and use these structures to describe the admissible positions and obtain the required shape-Wilf-equivalences.

In Section \ref{matching}, we study pattern avoidance in perfect matchings, 
continuing our investigation from \cite{SUK-arxiv}. While enumeration of matchings 
in  \cite{SUK-arxiv}  were obtained mainly through recurrence relations and bijections with ternary trees, the present paper also yields several lattice-path interpretations. In Theorem \ref{2nCn}, we construct a bijection between connected pattern-avoiding matchings and lattice paths with up-steps $U=(1,1)$ and down-steps $D=(1,-1)$. In Theorem \ref{large schroeder}, the corresponding matchings are related to Schr\"oder paths with steps
$$
U=(1,1),\qquad D=(1,-1),\qquad H=(2,0),
$$
and in Theorem \ref{catalan}, we obtain a bijection with Dyck paths. These constructions provide combinatorial interpretations for several of the matching enumeration sequences associated with the shape-Wilf-equivalence classes considered here.

Finally, in Section \ref{sec-set}, we determine the permutations that can arise as transversals of Ferrers boards avoiding a prescribed quadruple or quintuple pattern set. This analysis follows the same general framework as in  Part I \cite{SUK-arxiv}. To avoid unnecessary repetition, we give the essential structural arguments and refer to  Part I \cite{SUK-arxiv} for the common details of the construction.

% \section{Classification of Shape Wilf Equivalent Classes for quadruples of Length-$3$ Patterns}\label{SWE-4}

\subsection{Notation}
\label{subsec:notn}
We use the same notation as in Part I. For convenience, we recall below the notation that will be used repeatedly throughout this paper.
\begin{table}[H]
\centering
\begin{tabular}{c|l}
\hline
Notation & Meaning\\
\hline
$[n]$ & $\{1,2,\ldots,n\}$\\
$\mathfrak{S}_n$ & the set of permutations of $[n]$\\
$\SSS_F(P)$ & the set of $P$-avoiding transversals of a Ferrers board $F$\\
$\SSS_n(P)$ & the set of permutations in $\mathfrak{S}_n$ avoiding every pattern in $P$\\
$\mathcal{F}_n$ & the set of Ferrers boards with $n$ rows and $n$ columns\\  
& containing the staircase board\\
$\mathcal{T}_F$ & the set of transversals of $F$\\
$\mathcal{T}_F(P)$ & the set of $P$-avoiding transversals of $F$\\
$\mathcal{M}_n$ & the set of perfect matchings on $[2n]$\\
$\mathcal{M}_n(P)$ & the set of perfect matchings on $[2n]$ avoiding $P$\\
$\mathcal{M}_n^c(P)$ & the set of connected perfect matchings on $[2n]$ avoiding $P$\\
$\sim_s$ & shape-Wilf-equivalence\\
\hline
\end{tabular}
\caption{Notation used throughout the paper.}
\label{tab:notation}
\end{table}

\section{Classification of Shape Wilf Equivalent Classes for Quadruples of Length-3 Patterns}
\label{SWE-4}

% \begin{table}[!htbp]
%     \centering
% \begin{tabular}{|l|l|l|l|l|l|}
% \noalign{\hrule height 1pt}
% \textbf{Class }&~~~~~~~\textbf{Shape-Wilf-equivalent triples} \\
% %none of them is in OEIS
% \noalign{\hrule height 1pt}

%  & $\{123,132,213,231\}\sim_s \{123,132,213,312\}\sim_s$  \\
%  I& $\{132,213,231,312\}\sim_s \{132,231,312,321\}\sim_s$ \\
% & $\{123,213,231,312\}\sim_s \{132,213,312,321\}\sim_s$\\
% & ~~~~~~~~~~~~~~~~~~~$\{132,213,231,321\}$\\
% \noalign{\hrule height 1pt}

% & $\{123,132,231,321\}\sim_s \{123,132,312,321\}\sim_s$   \\
% II&$\{123,213,312,321\}\sim_s\{123,213,231,321\}\sim_s$ \\

% \noalign{\hrule height 1pt}

% III & $\{123,132,231,312\} \sim_s\{213,231,312,321\}$ \\

% \noalign{\hrule height 1pt}
%  IV & ~~~~~~~~~~~~~~~~~~~$\{123,132,213,321\}$ \\
% \noalign{\hrule height 1pt}
% V & ~~~~~~~~~~~~~~~~~~~$\{123,231,312,321\}$ \\

% \noalign{\hrule height 1pt}

% \end{tabular}
% \caption{The shape-Wilf-equivalent classes for quadruples of patterns of length $3$}
% \label{table2}
% \end{table}

In this Section, we determine all shape-Wilf-equivalence 
classes of quadruple pattern sets of length $3$. Our proofs 
are obtained using two different approaches. For several 
equivalence classes, the 
encoding techniques developed in 
\cite{SUK-arxiv} remain applicable with suitable modifications. These arguments yield explicit bijections 
between the corresponding families of pattern-avoiding transversals. 
For the remaining cases, we introduce a different approach 
based on Dyck paths associated with Ferrers boards containing 
the staircase shape. By studying certain distinguished cells and structural properties of these Dyck paths, we obtain not only the required shape-Wilf-equivalences, but also an explicit description of the ``suitable cells'' or ``admissible positions'' of transversals avoiding the given pattern sets in each Ferrers board. A cell $c$ in the current top row is $P$-admissible if there exists a completion of the current partial transversal to a $P$-avoiding transversal after placing a $1$ in $c$.

Enumerative consequences of our classifications are 
studied in Section \ref{matching}, where we consider 
pattern avoidance in perfect matchings. Since shape-Wilf-equivalent pattern sets give rise to the same number of avoiding transversals and matchings, each equivalence class corresponds to a common counting sequence. Table \ref{table2} summarizes all shape-Wilf-equivalence classes of quadruples together with the corresponding enumeration sequences, references, and OEIS entries whenever available.

For convenience of notation, we assign labels to the quadruple pattern sets considered throughout this paper. Since many arguments involve repeated references to these sets, the notation below in Table \ref{table4} will simplify the presentation of proofs and equivalence classes.

% \begin{table}[H]
% \centering
% \begin{tabular}{c|c}
% Notation & Pattern set\\
% \hline
% $P_1$ & $\{123,132,213,231\}$\\
% $P_2$ & $\{123,132,213,312\}$\\
% $P_3$ & $\{123,132,213,321\}$\\
% $P_4$ & $\{123,132,231,312\}$\\
% $P_5$ & $\{123, 132, 231, 321\}$\\
% $P_6$ & $\{123, 132, 312, 321\}$\\
% $P_7$ & $\{132, 213, 231, 312\}$\\
% $P_8$ & $\{132, 213, 231, 321\}$\\
% $P_9$ & $\{132, 231, 312, 321\}$\\
% $P_{10}$ & $\{213,231,312,321\}$\\
% $P_{11}$ & $\{123, 213, 231, 312\}$\\
% $P_{12}$ & $\{123, 213, 231, 321\}$\\
% $P_{13}$ & $\{123,231,312,321\}$ \\
% $P_{14}$ & $\{123, 213, 312, 321\}$\\
% $P_{15}$ & $\{132, 213, 312, 321\}$\\
% \end{tabular}
% \caption{Quadruple pattern sets considered in this paper.}
% \label{table4}
% \end{table}

\begin{table}[H]
\centering
\begin{tabular}{c|c|c|c}
Notation & Pattern set & Notation & Pattern set\\
\hline
$P_1$ & $\{123,132,213,231\}$& $P_9$ & $\{132, 231, 312, 321\}$\\
$P_2$ & $\{123,132,213,312\}$ & $P_{10}$ & $\{213,231,312,321\}$\\
$P_3$ & $\{123,132,213,321\}$ & $P_{11}$ & $\{123, 213, 231, 312\}$ \\
$P_4$ & $\{123,132,231,312\}$ & $P_{12}$ & $\{123, 213, 231, 321\}$\\ 
$P_5$ & $\{123, 132, 231, 321\}$ & $P_{13}$ & $\{123,231,312,321\}$ \\
$P_6$ & $\{123, 132, 312, 321\}$ & $P_{14}$ & $\{123, 213, 312, 321\}$\\
$P_7$ & $\{132, 213, 231, 312\}$ & $P_{15}$ & $\{132, 213, 312, 321\}$\\
$P_8$ & $\{132, 213, 231, 321\}$\\
\end{tabular}
\caption{Quadruple pattern sets considered in this paper.}
\label{table4}
\end{table}

\begin{table}[H]
    \centering
    \scriptsize{
    {\renewcommand{\arraystretch}{1.5}
\begin{tabular}{|l|l|l|l|l|l|}
\noalign{\hrule height 1pt}
\textbf{Class }&\textbf{Shape-Wilf-equivalent quadruples  $\mathbf{P}$}& ~~~$\mathbf{|\mathcal{M}_n(P)|~(n=1,2,\ldots,8,\ldots)}$ & ~~~~~~\textbf{Ref and OEIS}\\

%none of them is in OEIS
\noalign{\hrule height 1pt}

\hyperref[quadI]{I} & ~~~~~~~~~~~$P_1\sim_s P_2\sim_s P_7 \sim_s P_9$ &  $1, 3, 11, 43, 173, 707, 2917, 12111,\ldots$  & Theorem \ref{2nCn_rec}, ~\href{https://oeis.org/A026671}{A026671}\\
 & ~~~~~~~~~~~ $\sim_s P_{8}\sim_s P_{11}\sim_s P_{15}$ & & \\

\noalign{\hrule height 1pt}

 \hyperref[quadII]{II} &~~~~~~~~~ $P_5\sim_s P_6\sim_sP_{12}\sim_s P_{14}$  &$1,3,11,41,153,571,2131,7953 \ldots$& Theorem \ref{linear_quadruple1}, ~\href{https://oeis.org/A001835}{A001835}
\\

\noalign{\hrule height 1pt}

\hyperref[quadIII]{III} & ~~~~~~~~~~~~~~~~~~~~~$P_4 \sim_s P_{10}$ & $1, 3, 11, 45, 197, 903, 4279, 20793, \ldots$ &Theorem \ref{little schroeder}, ~\href{https://oeis.org/A001003}{A001003}\\

\noalign{\hrule height 1pt}
IV & ~~~~~~~~~~~~~~~~~~~~~~~~~~~$P_3$ & $1,3,11,42,161,618,2372,9105,\ldots$& Theorem \ref{linear_quadruple2}\\
 
\noalign{\hrule height 1pt}
V & ~~~~~~~~~~~~~~~~~~~~~~~~~~~$P_{13}$ & $1, 3, 11, 42, 162, 627, 2430, 9423, \ldots$& Theorem \ref{linear_quadruple3}, ~\href{https://oeis.org/A122368}{A122368}\\

\noalign{\hrule height 1pt}
\end{tabular}
}}
\caption{The shape-Wilf-equivalent classes for quadruples of patterns of length $3$}
\label{table2}
\end{table}

% \subsection{Class I (\texorpdfstring{$\mathbf{P_1 \sim_s P_2 \sim_s P_{7}
% \sim_s P_{8} \sim_s P_{9} \sim_s P_{11} \sim_s P_{15}}$}
% {P1 ~s P2 ~s P7 ~s P8 ~s P9 ~s P11 ~s P15})}
\subsection{Class I}\label{quadI}
% \subsection{Class I ($\mathbf{P_1 \sim_s P_2 \sim_s P_7 \sim_s P_8 \sim_s P_9 \sim_s P_{11}\sim_s P_{15}}$)}

In this section, we classify the shape-Wilf-equivalence class I listed in Table \ref{table2}. This class consists of seven quadruple pattern sets $P_1,P_2,P_7,P_8,P_9,P_{11},P_{15}$.
The arguments are divided into two parts. The first part uses the encoding techniques developed in Section \ref{quad-27811} to establish several equivalences. In Section \ref{quad-1215}, we use structural properties of the associated Dyck paths to complete the classification.

At first we discuss all admissible positions for placing a $1$ for all pattern sets in Class I. 
\begin{lemma}\label{position-class1}
    At any stage of the process, if the topmost row contains exactly two white cells, then either cell may receive a $1$. However, if the row contains three or more white cells, then it admits at most two admissible positions for placing a $1$ while avoiding the pattern set.

    \begin{itemize}
        \item For the pattern sets $P_1=\{123,132,213,231\}$ and $P_9=\{132,231,312,321\}$, the topmost row admits a unique admissible position, namely the first and last cell respectively, unless it contains exactly two white cells. 
        
        % and the rightmost column has no white cells above, in which case both positions are valid.

        \item For the pattern sets $P_2=\{123,132,213,312\}$  and  $P_{11}=\{123,213,231,312\}$,  the only valid positions are the first and second cells in that row.
        
        \item For the pattern sets $P_7= \{132,213,231,312\}$ and  $P_8=\{132,213,231,321\}$, the only valid positions are the first and last cells in that row.
        % \item If $P=\{132,231,312,321\}$\red{(9)},  the topmost row admits a unique admissible position, namely the last cell, unless it contains exactly two white cells and the rightmost column has no white cells above, in which case both positions are valid.
        \item For the pattern set $P_{15}=\{132,213,312,321\}$, the only valid positions are the last two cells in that row.
        
    \end{itemize}
\end{lemma}
% \subsubsection{\texorpdfstring{$\mathbf{  P_2 \sim_s P_{7}
% \sim_s P_{8}  \sim_s P_{11} \sim_s P_{15}}$}
% {  P2 ~s P7 ~s P8 ~s P11 ~s P15}}
\subsubsection{Equivalence of $P_2,P_7,P_8,P_{11},P_{15}$}
\label{quad-27811}
%$\{123,132,213,312\} \red{(2)}\sim_s \{132,213,231,312\}\red{(7)}\sim_s \{132,213,231,321\}\red{(8)}\sim_s \{123,213,231,312\}$\red{(11)}
\begin{lemma}\label{p2p7p8p11}
     Let $F$ be the Ferrers board. At some stage $i$ of the respective insertion processes for $P$, let $F'$ denote the board consisting of the remaining white cells of $F$. Suppose that the top row of $F'$ corresponds to row $r$ and the columns $c_1<\dots<c_l$ of $F$. Then in each of the following situations there is at most one suitable position for placing the next $1$:
     \begin{enumerate}
        \item[Case 1:] If a $1$ occurs strictly above and to the right of $(c_l,r)$; then the only valid position is  $c_l$ for $P_7=\{132,213,231,312\}$, $P_8=\{132,213,231,321\}$ and $P_{15}=\{132, 213, 312, 321\}$.
        \item[Case 2:] If the rectangle with vertices $(1,r+1)$, $(c_1-1,r+1)$, $(c_1-1,\ell(c_2))$, $(1,\ell(c_2))$ contains a $1$; then the only valid position is $c_1$ for $P_2=\{123,132,213,312\}$, $P_7=\{132,213,231,312\}$, $P_{11}=\{123,213,231,312\}$ and $c_l$ for $P_8=\{132,213,231,321\}$. 
        \item[Case 3:] If a $1$ from stage $i-1$ lies in row $r+1$ between columns $c_k$ and $c_{k+1}$; then the only valid position is $c_{k}$ for $P_2=\{123,132,213,312\}$, $P_{15}=\{132, 213, 312, 321\}$ and $c_{k+1}$ for $P_{11}=\{123,213,231,312\}$.
     \end{enumerate}
\end{lemma}
% \paragraph{Encoding Process:}
%  Define an encoding for transversals that avoid $P$.
% Let $T$ be such a transversal of $F$, and let 
% $w = w_1 \dots w_n$ be the associated word over $\{0,1,2\}$.  
% For the $1$ in the $i$-th row (from the top), set:  
% \begin{itemize}
%    \item $w_i=1$ if it is in the leftmost position of the remaining white cells and its placement is not forced. 
%    \item $w_i=2$ if it is in the rightmost position of the remaining white cells and its placement is not forced. 
%     \item $w_i = 0$ if the placement is forced.
% \end{itemize}

\paragraph{Encoding process.}
Let $P\in\{P_2,P_7,P_8,P_{11},P_{15}\},$
and let $T$ be a $P$-avoiding transversal of a Ferrers board $F$. We associate with $T$ a word $w=w_1\cdots w_n\in\{0,1,2\}^n.
$ Suppose that, at the $i$-th stage, the remaining white cells in the current top row lie in columns $c_1<c_2<\cdots<c_\ell.$
When the placement is not forced, the two admissible positions are ordered from left to right and denoted by $a_i<b_i$. More explicitly,

$$
(a_i,b_i)=
\begin{cases}
(c_1,c_2), & \text{if }P=P_2\text{ or }P=P_{11},\\[2mm]
(c_1,c_\ell), & \text{if }P=P_7\text{ or }P=P_8,\\[2mm]
(c_{\ell-1},c_\ell), & \text{if }P=P_{15}.
\end{cases}
$$

We then define

$$
w_i=
\begin{cases}
0, & \text{if the placement is forced},\\
1, & \text{if the $1$ is placed in }a_i,\\
2, & \text{if the $1$ is placed in }b_i.
\end{cases}
$$

If the current row contains exactly two white cells, then $a_i=c_1$ and $b_i=c_2$ for all five pattern sets.

\begin{lemma}\label{i+1-forced}
Let $P\in\{P_2,P_7,P_8,P_{11},P_{15}\}.$
Suppose that, at the $i$-th step of the encoding process, the current top row contains at least three white cells. If the $i$-th step is unforced, then the $(i+1)$-st step is forced. Equivalently, two consecutive steps cannot both be unforced unless the $i$-th step contains exactly two white cells.
\end{lemma}

\begin{proof}
For $P_2$ and $P_{11}$, the result follows directly from
Lemmas \ref{position-class1} and \ref{p2p7p8p11}. We therefore consider
$P_7,P_8$, and $P_{15}$.
First let $P=P_7=\{132,213,231,312\}.$
Suppose that the $i$-th step is unforced and that the current top row contains at least three white cells in columns
$c_1<c_2<\cdots<c_\ell.$
By Lemma \ref{position-class1}, the two admissible positions are $c_1$ and $c_\ell$.
Assume first that the $1$ is placed in $c_1$, so that $w_i=1$.
If the $(i+1)$-st step contains only the columns $c_2<\cdots<c_\ell,$
then its placement is forced by Lemma \ref{p2p7p8p11}. Otherwise, suppose that the next row contains additional white cells to the right, say
$$c_2<\cdots<c_\ell<\cdots<c_{\ell'}.$$
We claim that the next placement is forced. Indeed, if the next $1$ were placed in the other admissible position $c_{\ell'}$, then the columns $c_2$ and $c_3$ would subsequently have to be filled in increasing order in order to avoid $213$. Together with the $1$ already placed in $c_1$, the columns $c_1,c_2,c_3$ would then form an occurrence of $312$, a contradiction. Hence the $(i+1)$-st step is forced.

If instead the second admissible position $c_\ell$ is chosen at the $i$-th step, so that $w_i=2$, the analogous argument using the other possible placement at the $(i+1)$-st step again produces one of the forbidden patterns in $P_7$. Hence the $(i+1)$-st step is forced in this case as well. Therefore,

$$
w_i\neq0\quad\Longrightarrow\quad w_{i+1}=0
$$

whenever the $i$-th step contains at least three white cells.

The proof for $P_8=\{132,213,231,321\}$
is analogous: by Lemma~\ref{position-class1}, the only possible unforced choices are the first and last admissible positions, and either choice at step $i$ forces the placement at step $i+1$ by Lemma~\ref{p2p7p8p11} together with the avoidance of the patterns in $P_8$.

Finally, consider $P=P_{15}=\{132,213,312,321\}.$
Suppose that the $i$-th step is unforced and that its white cells lie in columns
$$
c_1<c_2<\cdots<c_\ell,
\qquad \ell\ge3.$$
By Lemma~\ref{position-class1}, the two admissible positions are $c_{\ell-1}$ and $c_\ell$.

First suppose that the $1$ is placed in $c_{\ell-1}$, so that $w_i=1$.
If the $(i+1)$-st step contains only
$$c_1<\cdots<c_{\ell-2}<c_\ell,$$ then the next placement is forced by Lemma \ref{p2p7p8p11}. Otherwise, suppose that the next row contains additional white cells to the right:
$$c_1<\cdots<c_{\ell-2}<c_\ell<\cdots<c_{\ell'}.$$
If the next step were also unforced, its two possible positions would be
$c_{\ell'-1}$ and $c_{\ell'}$. However, the columns $c_1$ and $c_\ell$ must subsequently be filled in increasing order in order to avoid $213$. Together with the $1$ in $c_{\ell-1}$, this would create an occurrence of $132$, a contradiction. Thus the $(i+1)$-st step is forced.

If the second admissible position $c_\ell$ is chosen at the $i$-th step, the same argument, applied to the remaining admissible position at the next step, again yields a forbidden pattern from $P_{15}$. Hence the $(i+1)$-st step is forced in this case as well.

Therefore, for each $P\in\{P_2,P_7,P_8,P_{11},P_{15}\},$ an unforced step containing at least three white cells is necessarily followed by a forced step. This proves the result.
\end{proof}

% \subsubsection{\texorpdfstring{$\mathbf{P_1 \sim_s P_2 \sim_s P_{9}}$}
% {P1 ~s P2  ~s P9 }}
\subsection{Equivalence of $P_1,P_2,P_9$}
\label{quad-1215}

\begin{note}\label{note: ferrers-dyck}
Let $F$ be a Ferrers board contained in the $n \times n$ square, containing the staircase board and drawn in French notation.
We associate to $F$ a Dyck path $D$ of semilength $n$ as follows.

Place $F$ inside the square $[0,n]\times[0,n]$ and draw the main diagonal 
from $(0,n)$ to $(n,0)$.  
The boundary of $F$ lying weakly above the diagonal determines a lattice path
from $(0,0)$ to $(2n,0)$ consisting of up-steps $U=(1,1)$ and down-steps
$D=(1,-1)$, obtained by reading the boundary from left to right.
This path never goes below the $x$-axis and hence is a Dyck path.
This correspondence is bijective.
\end{note}
\begin{definition}
Let $D$ be a Dyck path.  
A \emph{run} of $D$ is a maximal subpath of the form
\[
U^{a}D^{b}, \qquad a,b \ge 1,
\]
that is, a maximal consecutive block of up-steps followed by a maximal
consecutive block of down-steps.
A new run begins whenever a down-step is followed by an up-step.
\end{definition}

\begin{definition}\label{distinguished}
Let $D$ be the Dyck path associated with a Ferrers board $F$.
Consider a run of $D$, and let $d$ be its last down-step and $u$ its first up-step.
\begin{itemize}
    \item If the height of $D$ immediately before the step $d$ is $1$, then we associate with this run the unique cell of $F$ lying immediately above the diagonal at the corresponding position. This cell is called a \emph{right distinguished cell}.

    \item If the height of $D$ immediately after the step $u$ is $1$, then we associate with this run the unique cell of $F$ lying immediately above the diagonal at the corresponding position. This cell is called a \emph{left distinguished cell}.
\end{itemize}
\end{definition}
\begin{note}
A run contributes a right distinguished cell only when the height of $D$ immediately before its last down-step is $1$. Similarly, it contributes a left distinguished cell only when the height of $D$ immediately after its first up-step is $1$.
\end{note}

In Figure \ref{Fig11}, we illustrate the correspondence between a Ferrers board and its associated Dyck path. The distinguished cells are indicated on the path, with right distinguished cells marked by blue dots and left distinguished cells marked by red dots.
%and green dots are for $P_{15}=\{132,213,312,321\}$.

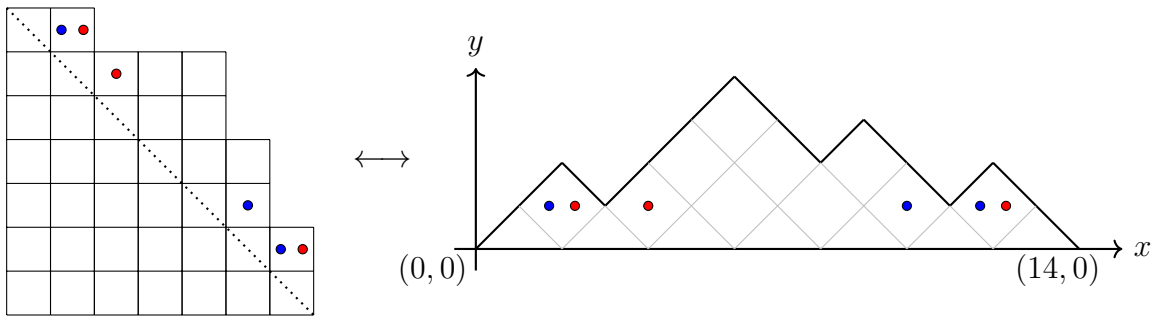
\begin{figure}[H]

   \centering
{
\begin{minipage}{.28\textwidth}
\begin{tikzpicture}[scale=0.58]
    
    \draw (0,0) grid (1,7);
    \draw (1,0) grid (2,7);
    \draw (2,0) grid (3,6);
    \draw (3,0) grid (4,6);
    \draw (4,0) grid (5,6);
    \draw (5,0) grid (6,4);
    \draw (6,0) grid (7,2);

    \draw[dotted][thick] (0,7) -- (7,0);

    \filldraw [fill=blue]  (1.25,6.5) circle (3pt); 
    \filldraw  [fill=blue] (5.5,2.5) circle (3pt); 
    \filldraw  [fill=blue] (6.25,1.5) circle (3pt); 

    \filldraw [fill=red]  (1.75,6.5) circle (3pt); 
    \filldraw  [fill=red] (2.5,5.5) circle (3pt); 
    \filldraw  [fill=red] (6.75,1.5) circle (3pt); 

    % \filldraw [fill=green]  (1.5,6.25) circle (3pt);
    % \filldraw  [fill=green] (5.5,3.5) circle (3pt); 
    % \filldraw  [fill=green] (6.5,1.25) circle (3pt); 
    
\end{tikzpicture}
\end{minipage}
\begin{minipage}{.02\textwidth}
   $\longleftrightarrow$
\end{minipage}
\begin{minipage}{.52\textwidth}
\begin{tikzpicture}[scale=0.57]
\node[] at (-1,-0.5)   { $(0,0)$};
\node[] at (13.5,-0.5)   { $(14,0)$};
\draw[->, thick] (-0.5,0) -- (15,0) node[right] {$x$};
\draw[->, thick] (0,-0.5) -- (0,4.2) node[above] {$y$};

 \draw [thick](0,0) -- (2,2);
 \draw [thick](2,2) -- (3,1);
 \draw [thick](3,1) -- (6,4);
 \draw [thick](6,4) -- (8,2);
 \draw [thick](8,2) -- (9,3);
 \draw [thick](9,3) -- (11,1);
 \draw [thick](11,1) -- (12,2);
 \draw [thick](12,2) -- (14,0);
 
\draw[ gray!50] (1,1) -- (2,0);
\draw[ gray!50] (3,1) -- (2,0);
\draw[ gray!50] (3,1) -- (4,0);
\draw[ gray!50] (5,1) -- (4,0);
\draw[ gray!50] (5,1) -- (4,2);
\draw[ gray!50] (5,1) -- (6,0);
\draw[ gray!50] (5,1) -- (6,2);
\draw[ gray!50] (6,2) -- (5,3);
\draw[ gray!50] (6,2) -- (7,1);
\draw[ gray!50] (6,2) -- (7,3);
\draw[ gray!50] (7,1) -- (8,2);
\draw[ gray!50] (7,1) -- (6,0);
\draw[ gray!50] (7,1) -- (8,0);
\draw[ gray!50] (9,1) -- (8,0);
\draw[ gray!50] (9,1) -- (8,2);
\draw[ gray!50] (9,1) -- (10,2);
\draw[ gray!50] (9,1) -- (10,0);
\draw[ gray!50] (11,1) --(10,0);
\draw[ gray!50] (11,1) -- (12,0);
\draw[ gray!50] (13,1) -- (12,0);

\filldraw [fill=blue] (1.7,1) circle (3pt);
\filldraw  [fill=blue](10,1) circle (3pt);
\filldraw  [fill=blue](11.7,1) circle (3pt);

\filldraw [fill=red] (2.3,1) circle (3pt);
\filldraw  [fill=red](4,1) circle (3pt);
\filldraw  [fill=red](12.3,1) circle (3pt);

% \filldraw [fill=green] (2,0.75) circle (3pt);
% \filldraw  [fill=green](9,2) circle (3pt);
% \filldraw  [fill=green](12,0.75) circle (3pt);

\end{tikzpicture}
\end{minipage}
}
\caption{Distinguished cells }  
\label{Fig11}
 \end{figure}

\begin{lemma}\label{right-distinguished}
Let \(F\) be a Ferrers board and let \(D\) be its associated Dyck path. Then the \(i\)-th row of \(F\) contains a right distinguished cell if and only if, at the \(i\)-th step of the encoding process, the current top row contains exactly two white cells.
Consequently, for
$$P_1=\{123,132,213,231\}
\qquad\text{and}\qquad
P_9=\{132,231,312,321\},$$
we have $P_1\sim_s P_9.$
\end{lemma}

\begin{proof}
By the correspondence between the boundary of \(F\) and the Dyck path \(D\), a right distinguished cell occurs precisely when the boundary returns to height \(1\) at the end of a run. This is equivalent to saying that, at the corresponding stage of the encoding process, exactly two white cells remain in the current top row.

By Lemma~\ref{position-class1}, for both \(P_1\) and \(P_9\), the placement is forced whenever the current row contains at least three white cells, while both positions are admissible when exactly two white cells remain. Hence the unforced steps are exactly the rows containing right distinguished cells.

Therefore, each right distinguished cell contributes two choices for both pattern sets, whereas every other row contributes a unique choice. Thus, for every Ferrers board \(F\),
$$|\SSS_F(P_1)|=|\SSS_F(P_9)|.$$
Hence $P_1\sim_s P_9.$
\end{proof}

\begin{lemma}\label{p1p2}
Let $T$ be a transversal of a Ferrers board $F$ that avoids a pattern set $P$, and let
$w=w_1\cdots w_n$ be its encoding over the alphabet $\{0,1,2\}$. Then the following hold. 
\begin{itemize}
    \item[(i)] For $P_1=\{123,132,213,231\}$, the right distinguished cells of $F$ are exactly the positions at which the choice $w_i=2$ is permitted. Hence, for a fixed Ferrers board $F$, the number of transversals avoiding $P_1$ is
    $$2^{\#\{\text{right distinguished cells of }F\}}.$$
    \item[(ii)] For $P_2=\{123,132,213,312\}$, the left distinguished cells of $F$ are exactly the positions at which the choice $w_i=2$ is permitted. Hence, for a fixed Ferrers board $F$, the number of transversals avoiding $P_2$ is
    $$2^{\#\{\text{left distinguished cells of }F\}}.$$
\end{itemize}
\end{lemma}
\begin{proof}
    \begin{itemize}
        \item[(i)] By Lemma~\ref{position-class1}, if a row contains at least three white cells, then the topmost row admits a unique admissible position, namely the leftmost cell. Thus, non-forced choices occur only in rows containing exactly two white cells. Let $m$ denote the number of such rows arising during the encoding process. Then, for a fixed Ferrers board $F$, the number of transversals avoiding $P_1$ is exactly $2^m$.

        By the definition of right distinguished cells, such cells occur precisely in those rows where the encoding process leaves exactly two white cells, with the right distinguished cell being the rightmost admissible position. Hence, $$m=\#\{\text{right distinguished cells of }F\}.$$
    Therefore,$$|\SSS_F(P_1)|=2^{\#\{\text{right distinguished cells of }F\}}.$$ 

    \item[(ii)] Now consider the pattern set $P_2=\{123,132,213,312\}.$
    By Lemma~\ref{position-class1}, if a row contains at least two white cells, then the only admissible positions are the first and second white cells of that row.
    Further, by Lemma~\ref{p2p7p8p11}, let $F$ be a Ferrers board, and let $F'$ denote the board of remaining white cells at some stage of the insertion process. Suppose the top row of $F'$ corresponds to row $r$ and occupies columns $c_1<\cdots<c_l.$
Then the placement is forced in either of the following situations:
\begin{itemize}
    \item the rectangle with vertices $
    (1,r+1),\ (c_1-1,r+1),\ (c_1-1,\ell(c_2)),\ (1,\ell(c_2))$ contains a $1$, in which case the only admissible position is $c_1$;
    
    \item a $1$ from stage $i-1$ lies in row $r+1$ between columns $c_k$ and $c_{k+1}$, in which case the only admissible position is $c_k$.
\end{itemize}
Hence non-forced choices occur precisely in rows where the first two white cells are both admissible, equivalently when the first white cell has no cell in its upper-right corner. Let $m$ denote the number of such rows.

By the definition of left distinguished cells, these are exactly the cells with no cell in the upper-right corner and which occur as the first white cell of their row. Therefore,
$$m=\#\{\text{left distinguished cells of }F\}.$$
Consequently,
$$|\SSS_F(P_2)|=2^{\#\{\text{left distinguished cells of }F\}}.$$
    \end{itemize}
\end{proof}
\begin{note}
    For any Ferrers board $F$, $\#$\{left distinguished cells of $F$\}= $\#$\{right distinguished cells of $F$\}. Therefore, for any Ferrers board $F$,
$|\SSS_F(P_1)|=|\SSS_F(P_{2})|,$
and hence
$$P_1\sim_s P_{2}.$$
\end{note}

\subsection{Class II}\label{quadII}

Class II contains the following four sets: $P_5,P_6,P_{12},P_{14}$. We will show that all four are shape-Wilf-equivalent.

% In this section we will prove the shape-Wilf-equivalence between the set of patterns $\{123,132,231,321\}$, $\{123,132,312,321\}$, $\{123,213,231,321\}$  and $\{123,213,312,321\}$.
\begin{lemma}\label{atmost3}
    Let $P=P_5,P_6,P_{12},P_{14}$, if the $i$-th row of a Ferrers board has more than three white cells, then the board admits no transversal that avoids $P$.
\end{lemma}
\begin{proof}
     We prove the statement for the pattern set
$$P_5=\{123,132,231,321\};$$
the remaining cases follow by analogous arguments.

Suppose the $i$-th row contains at least four white cells in columns
$$c_1<c_2<c_3<c_4.$$
Since the patterns $123$, $132$, $231$, and $321$ are forbidden, the only possible relative orders among three chosen columns are $213$ and $312$.

Consider placing a $1$ in one of the available cells of this row. If the $1$ is placed in column $c_1$, then the columns $c_1,c_3,c_4$ necessarily determine the pattern $123$. Similarly, placing a $1$ in column $c_2$ forces the columns $c_2,c_3,c_4$ to form the pattern $123$.
On the other hand, if the $1$ is placed in column $c_3$, then the columns $c_1,c_2,c_3$ form the pattern $321$. Likewise, choosing column $c_4$ forces the columns $c_1,c_2,c_4$ to form the pattern $321$.

Thus every possible placement of a $1$ in the $i$-th row creates a forbidden pattern belonging to $P_5$. Consequently, the Ferrers board admits no transversal avoiding the set $P_5$.
\end{proof}

\begin{lemma}\label{position-class2}
    At any stage of the process, if the topmost row contains exactly two white cells, then either cell may receive a $1$. However, if the row contains three cells, then it admits at most two admissible positions for placing a $1$ while avoiding the pattern set $P$.

    \begin{itemize}
        \item For $P_5=\{123,132,231,321\}$, at the $i$-th step the only valid positions for placing a $1$ in the topmost row are the first and third cells. The only exception occurs when both the $i$-th and $(i+1)$-th steps have topmost rows containing exactly three white cells; in this case, to avoid $P$, the $1$ in the topmost row must be placed in the third column.

        \item For $P_6=\{123,132,312,321\}$, the only valid positions are the second and third cells in that row.
        \item If $P_{12}= \{123,213,231,321\}$,  the only valid positions are the first and second cells in that row.
        
        \item For $P_{14}=\{123,213,312,321\}$,  the topmost row admits a unique admissible position, namely the second cell, unless it contains exactly two white cells. In that case both the cells are admissible.
        
    \end{itemize}
\end{lemma}
% \subsubsection{\texorpdfstring{$\mathbf{P_6 \sim_s P_{12}}$}
% {P6 ~s P12}}
\subsubsection{Equivalence of $P_6, P_{12}$}\label{quad-612}
%$\{123,132,312,321\}\red{6}\sim_s\{123,213,231,321\}$\red{12}
\begin{lemma}
Let $F$ be a Ferrers board, and let $\ell(c)$ denote the height of column $c$. At some stage of the construction for $P$, let $F'$ be the subboard of remaining white cells. Suppose the top row of $F'$ corresponds to row $r$ of $F$ and occupies columns $c_1<c_2<c_3$. Then the next placement is uniquely determined in any of the following situations:

\begin{itemize}
    \item For $P_{12}=\{123,213,231,321\}$ and the rectangle with vertices $(1,r+1)$, $(c_1-1,r+1)$, $(c_1-1,\ell(c_2))$, $(1,\ell(c_2))$ contains a $1$; then the only valid position is $c_3$.
    \item   For  $P_6=\{123,132,312,321\}$ and a $1$ occurs strictly above and to the right of $(c_3,r)$; then the only valid position is $c_1$.
    \item A $1$ from stage $i-1$ lies in row $r+1$ between columns $c_k$ and $c_{k+1}$; then the only valid position is $c_{k}$ for $P_6=\{123, 132, 312, 321\}$ and $c_{k+1}$ for $P_{12}=\{123,213,231,321\}$. 
\end{itemize}
\end{lemma}
% \subsubsection{\texorpdfstring{$\mathbf{P_5 \sim_s P_ 6\sim_s P_{14}}$}
% {P5 ~s P6 ~s P14 }}\label{quad-5614}
\subsubsection{Equivalence of $P_5,P_6,P_{14}$}\label{quad-5614}
%$\{123,132,231,321\}\red{5} \sim_s \{123,132,312,321\}\red{6} \sim_s \{123,213,312,321\}$\red{14}
\begin{note}
    For any Ferrers board $F$ and any transversal $T$ avoiding $P$, the associated Dyck path $D$ has peaks of height at most $2$, by Lemma \ref{atmost3}.
\end{note}

\begin{definition}[Ferrers height of a peak]
Let $D$ be the Dyck path associated with a Ferrers board $F$. If a peak of $D$ has coordinates $(k,h+1),$
then we say that the peak has \emph{Ferrers height} $h$.

Equivalently, a peak has Ferrers height $h$ if the corresponding row of $F$, at that stage of the encoding process, contains exactly $(h+1)$ white cells.
\end{definition}
Thus, a peak of Ferrers height $1$ corresponds to a row containing exactly two white cells, while a peak of Ferrers height $2$ corresponds to a row containing exactly three white cells.

\begin{definition}{(\textit{Distance between peaks})}
   In a Dyck path $D$, two peaks of Ferrers height $h$ are said to be at distance $d$ if their coordinates are $(k,h+1)$ and $(k+2d,h+1)$ for some integer $k$.
\end{definition}

\begin{lemma}\label{for p5}
    Let $T$ be a transversal of a Ferrers board $F$ that avoids the pattern set
$P_5=\{123,132,231,321\}$, and let $D$ be the associated Dyck path.
Let $w=w_1\cdots w_n$ be the corresponding encoding over the alphabet $\{0,1,2\}$.
Then the positions at which the choice $w_i=2$ is admissible are exactly the following:
\begin{itemize}
    \item the peaks of $D$ of Ferrers height $1$;
    \item the peaks of $D$ of Ferrers height $2$ that are at distance at least $2$ from the next peaks of height $2$, whenever they exist;
    \item in any maximal block of consecutive peaks of Ferrers height $2$ separated by distance $1$, the last peak.
\end{itemize}
\end{lemma}
The proof is analogous to that of Lemma \ref{for P6}.

\begin{lemma}\label{for P6}
Let $T$ be a transversal of a Ferrers board $F$ that avoids the pattern set
$P_6=\{123,132,312,321\}$, and let $D$ be the associated Dyck path.
Let $w=w_1\cdots w_n$ be the corresponding encoding over the alphabet $\{0,1,2\}$.
Then the positions at which the choice $w_i=2$ is admissible are exactly the following:
\begin{itemize}
    \item the peaks of $D$ of Ferrers height $1$;
    \item the peaks of $D$ of Ferrers height $2$ that are at distance at least $2$ from the preceding and succeeding peaks of Ferrers height $2$, whenever they exist;
    \item in any maximal block of consecutive peaks of height $2$ separated by distance $1$, the first peak.
\end{itemize}
\begin{proof}
A peak of Ferrers height $1$ in the Dyck path $D$ corresponds to the $i$-th row of $F$ having exactly two white cells. In this case, both positions are admissible, and choosing the rightmost cell corresponds to setting $w_i=2$.

A peak of Ferrers height $2$ that is at distance at least $2$ from both the preceding and succeeding peaks of height $2$ implies that there are no white cells above the last two white cells in the $i$-th row. Consequently, both positions are admissible, and choosing the rightmost cell corresponds to $w_i=2$.

Now consider a maximal block of consecutive peaks of Ferrers height $2$ separated by distance $1$. For the first peak in such a block, there are no white cells above the last two white cells of the corresponding row, so both placements are initially admissible. Suppose we place a $1$ in the rightmost cell, say in column $l$. Then, in the $(i+1)$-th row, the remaining white cells occur in columns $l-1$ and $l+1$, which are the last two white cells of that row. To avoid the pattern $123$, the placement in the $(l+1)$-th column is forced, leaving no choice at this stage. The same argument applies inductively to all subsequent peaks in the block, so only the first peak in the block admits a choice corresponding to $w_i=2$.
\end{proof}

\end{lemma}

\begin{lemma}\label{for P14}
    Let $T$ be a transversal of a Ferrers board $F$ that avoids the pattern set $P_{14}=\{123,213,312,321\}$, and let $D$ be the associated Dyck path. Let $w=w_1\cdots w_n$ be the corresponding encoding over the alphabet $\{0,1,2\}$. Then the positions at which the choice $w_i=2$ is admissible are exactly the right distinguished cells (Definition \ref{distinguished}).
\end{lemma}
The proof is identical to that of Lemma \ref{p1p2}, using Lemma \ref{position-class2}.
\vspace{.5cm}

In figure \ref{Fig12}, we illustrate the correspondence between a Ferrers board and its associated Dyck path. The admissible cells for $w_i=2$ are indicated on the path; blue dots are for $P_5=\{123,132,231,321\}$, red dots are for $P_{6}=\{123,132,312,321\}$ and green dots are for $P_{14}=\{123,213,312,321\}$.

\begin{figure}[H]
\centering

\begin{minipage}{0.32\textwidth}
\centering
\resizebox{\linewidth}{!}{%
\begin{tikzpicture}[scale=0.48]

\draw (0,0) grid (1,9);
\draw (1,0) grid (2,9);
\draw (2,0) grid (3,9);
\draw (3,0) grid (4,8);
\draw (4,0) grid (5,6);
\draw (5,0) grid (6,6);
\draw (6,0) grid (7,4);
\draw (7,0) grid (8,4);
\draw (8,0) grid (9,2);

\draw[dotted,thick] (0,9) -- (9,0);

\filldraw [fill=blue]  (3.5,7.5) circle (3pt);
\filldraw [fill=blue]  (5.25,5.5) circle (3pt);
\filldraw [fill=blue]  (7.25,3.5) circle (3pt);
\filldraw [fill=blue]  (8.25,1.65) circle (3pt);

\filldraw [fill=red]   (2.5,8.5) circle (3pt);
\filldraw [fill=red]   (5.75,5.5) circle (3pt);
\filldraw [fill=red]   (7.75,3.5) circle (3pt);
\filldraw [fill=red]   (8.75,1.65) circle (3pt);

\filldraw [fill=green] (3.5,6.5) circle (3pt);
\filldraw [fill=green] (5.5,4.5) circle (3pt);
\filldraw [fill=green] (7.5,2.5) circle (3pt);
\filldraw [fill=green] (8.5,1.25) circle (3pt);

\end{tikzpicture}
}
\end{minipage}
\hfill
\begin{minipage}{0.05\textwidth}
\centering
$\longleftrightarrow$
\end{minipage}
\hfill
\begin{minipage}{0.58\textwidth}
\centering
\resizebox{\linewidth}{!}{%
\begin{tikzpicture}[scale=0.43]

\node at (-1.1,-0.7) {$(0,0)$};
\node at (17.5,-0.7) {$(18,0)$};
\node at (3,3.5) {$(3,3)$};
\node at (5,3.5) {$(5,3)$};
\node at (9,3.5) {$(9,3)$};
\node at (13,3.5) {$(13,3)$};
\node at (16,2.5) {$(16,2)$};

\draw[->,thick] (-0.5,0) -- (19,0) node[right] {$x$};
\draw[->,thick] (0,-0.5) -- (0,4.2) node[above] {$y$};

\draw[thick] (0,0) -- (3,3);
\draw[thick] (3,3) -- (4,2);
\draw[thick] (4,2) -- (5,3);
\draw[thick] (5,3) -- (7,1);
\draw[thick] (7,1) -- (9,3);
\draw[thick] (9,3) -- (11,1);
\draw[thick] (11,1) -- (13,3);
\draw[thick] (13,3) -- (15,1);
\draw[thick] (15,1) -- (16,2);
\draw[thick] (16,2) -- (18,0);

\draw[gray!50] (1,1) -- (2,0);
\draw[gray!50] (4,2) -- (2,0);
\draw[gray!50] (2,2) -- (4,0);
\draw[gray!50] (5,1) -- (4,0);
\draw[gray!50] (5,1) -- (4,2);
\draw[gray!50] (5,1) -- (6,0);
\draw[gray!50] (5,1) -- (6,2);
\draw[gray!50] (7,1) -- (6,0);
\draw[gray!50] (7,1) -- (8,0);
\draw[gray!50] (9,1) -- (8,0);
\draw[gray!50] (9,1) -- (8,2);
\draw[gray!50] (9,1) -- (10,2);
\draw[gray!50] (9,1) -- (10,0);
\draw[gray!50] (11,1) -- (10,0);
\draw[gray!50] (11,1) -- (12,0);
\draw[gray!50] (14,2) -- (12,0);
\draw[gray!50] (12,2) -- (14,0);
\draw[gray!50] (14,0) -- (15,1);
\draw[gray!50] (16,0) -- (15,1);
\draw[gray!50] (16,0) -- (17,1);

\filldraw [fill=blue]  (5,2) circle (3pt);
\filldraw [fill=blue]  (8.7,2) circle (3pt);
\filldraw [fill=blue]  (12.7,2) circle (3pt);
\filldraw [fill=blue]  (15.7,1.15) circle (3pt);

\filldraw [fill=red]   (3,2) circle (3pt);
\filldraw [fill=red]   (9.3,2) circle (3pt);
\filldraw [fill=red]   (13.3,2) circle (3pt);
\filldraw [fill=red]   (16.3,1.15) circle (3pt);

\filldraw [fill=green] (6,1) circle (3pt);
\filldraw [fill=green] (10,1) circle (3pt);
\filldraw [fill=green] (14,1) circle (3pt);
\filldraw [fill=green] (16,0.65) circle (3pt);

\end{tikzpicture}
}
\end{minipage}

\caption{Distinguished cells}
\label{Fig12}
\end{figure}
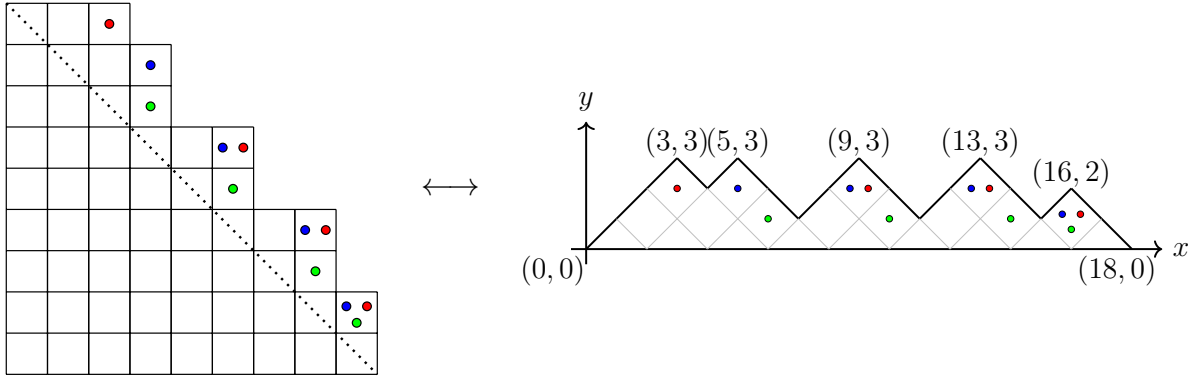
 
% \begin{note}
%     For the pattern sets $P_5$ and $P_6$, the admissible cells corresponding to the choice $w_i=2$ either coincide or occur at two consecutive peaks of Ferrers height $2$ separated by distance $1$. Hence, for every Ferrers board $F$, $|\SSS_F(P_5)|=|\SSS_F(P_6)|.$
%     Therefore, $$P_5\sim_s P_6.$$

%     Further, whenever the Dyck path has a peak of Ferrers height $1$, the right distinguished cell coincides with the admissible cell for $w_i=2$ in $P_5$. Otherwise, the right distinguished cell occurs immediately below the admissible cell for $w_i=2$ in $P_5$. Consequently, for every Ferrers board $F$, $|\SSS_F(P_5)|=|\SSS_F(P_{14})|.$
%     Hence, $$P_5\sim_s P_{14}.$$
% \end{note}

\begin{lemma}\label{P5P6P14-equivalence}
For every Ferrers board $F$,
$$|\SSS_F(P_5)|=|\SSS_F(P_6)|=|\SSS_F(P_{14})|.$$
Consequently, $P_5\sim_s P_6\sim_s P_{14}.$
\end{lemma}

\begin{proof}
First consider $P_5$ and $P_6$. By Lemmas \ref{for p5} and \ref{for P6}, the admissible cells corresponding to the choice $w_i=2$ coincide except within a maximal block of consecutive peaks of Ferrers height $2$ separated by distance $1$. In each such block, $P_5$ admits the choice $w_i=2$ at the last peak, whereas $P_6$ admits it at the first peak. Thus, every such block contributes exactly one unforced binary choice for each pattern set. Hence the two encodings have the same number of unforced choices, and therefore 
\begin{equation}
\label{eqn:p5-equal-p6}
    |\SSS_F(P_5)|=|\SSS_F(P_6)|.
\end{equation}

Now consider $P_5$ and $P_{14}$. By Lemma \ref{for P14}, the admissible cells corresponding to $w_i=2$ for $P_{14}$ are precisely the right distinguished cells. If the corresponding peak has Ferrers height $1$, this right distinguished cell coincides with the admissible cell for $w_i=2$ in the $P_5$-encoding. Otherwise, the right distinguished cell lies immediately below the corresponding admissible cell for $P_5$. Hence the admissible positions corresponding to $w_i=2$ for $P_5$ are in bijection with the right distinguished cells, and therefore with the admissible positions corresponding to $w_i=2$ for $P_{14}$. Thus, 
\begin{equation}
\label{eqn:p5-equal-p14}
    |\SSS_F(P_5)|=|\SSS_F(P_{14})|.
\end{equation}
%$$|\SSS_F(P_5)|=|\SSS_F(P_{14})|.$$
From \eqref{eqn:p5-equal-p6} and \eqref{eqn:p5-equal-p14}, we conclude that
$P_5\sim_s P_6\sim_s P_{14}.$
\end{proof}

% \subsection{Class III (\texorpdfstring{$\mathbf{P_4 \sim_s P_{10}}$}
% {P4 ~s P10 })}
\subsection{Class III}\label{quadIII}
\label{quad-410}
In this section we will prove the shape-Wilf-equivalence between the pattern sets  $P_{4},P_{10}$.
% $\{123,132,231,312\}$ and $\{213,231,312,321\}$.

\begin{lemma}
    At any stage of the process, the topmost row of white cells contains at most two positions where a $1$ can be placed while still avoiding the pattern set $P$.
    \begin{itemize}
        \item For $P_4=\{123,132,231,312\}$, then the only valid positions are the first and last cells in that row.
        \item For $P_{10}=\{213,231,312,321\}$, then the only valid positions are the last two cells in that row.
    \end{itemize}
\end{lemma}

\begin{lemma}
     Let $F$ be the Ferrers board. At some stage $i$ of the respective insertion processes for $P$, let $F'$ denote the board consisting of the remaining white cells of $F$. Suppose that the top row of $F'$ corresponds to row $r$ and the columns $c_1<\dots<c_l$ of $F$. Then in each of the following situations there is at most one suitable position for placing the next $1$:
     \begin{enumerate}
        \item[Case 1:] A $1$ occurs strictly above and to the right of $(c_l,r)$; then the only valid position is $c_1$ for $P=\{123,132,231,312\}$ and  $c_l$ for $P=\{213,231,312,321\}$.
        \item[Case 2:] The rectangle with vertices $(1,r+1)$, $(c_1-1,r+1)$, $(c_1-1,\ell(c_2))$, $(1,\ell(c_2))$ contains a $1$; then the only valid position is $c_1$ for $P=\{123,132,231,312\}$. This situation will not occur for $P=\{213,231,312,321\}$.
        \item[Case 3:] A $1$ from stage $i-1$ lies in row $r+1$ between columns $c_k$ and $c_{k+1}$; then the only valid position is $c_{k+1}$ for $P=\{213,231,312,321\}$. This situation will not occur for $P=\{123,132,231,312\}$.
     \end{enumerate}
\end{lemma}
% \paragraph{Encoding Process:}
% Define an encoding for transversals that avoid $P$.
% Let $T$ be such a transversal of $F$, and let 
% $w = w_1 \dots w_n$ be the associated word over $\{0,1,2\}$.  
% For the $1$ in the $i$-th row (from the top), set:  
% \begin{itemize}
%    \item $w_i=1$ if it is in the leftmost position of the remaining white cells and its placement is not forced. 
%    \item $w_i=2$ if it is in the rightmost position of the remaining white cells and its placement is not forced. 
%     \item $w_i = 0$ if the placement is forced.
% \end{itemize}

\paragraph{Encoding process.}
Let $P\in\{P_4,P_{10}\},$ and let $T$ be a $P$-avoiding transversal of a Ferrers board $F$. Suppose that, at the $i$-th step, the remaining white cells in the current top row lie in columns $c_1<c_2<\cdots<c_\ell.$
When the placement is not forced, denote the two admissible positions, from left to right, by $a_i<b_i$. Thus,
$$
(a_i,b_i)=
\begin{cases}
(c_1,c_\ell),&\text{if }P=P_4,\\[2mm]
(c_{\ell-1},c_\ell),&\text{if }P=P_{10}.
\end{cases}
$$
We associate with $T$ a word $
w=w_1\cdots w_n\in\{0,1,2\}^n$ by setting
$$
w_i=
\begin{cases}
0,&\text{if the placement at the $i$-th step is forced},\\
1,&\text{if the $1$ is placed in }a_i,\\
2,&\text{if the $1$ is placed in }b_i.
\end{cases}
$$

\begin{note}
   Although the lemmas and the encoding process for Class III are analogous to those for Class I, Lemma~\ref{i+1-forced} does not hold in this case. Indeed, if $w_i\neq 0$, then the $(i+1)$-st step need not be forced; it may or may not admit more than one admissible placement.

    For example, consider the Ferrers board $(5,5,5,3,3),$
where row lengths are listed from bottom to top. In this case, the first step is not forced, whereas the second step is forced. On the other hand, for the Ferrers board $(5,5,5,5,4),$ neither the first nor the second step is forced.
\end{note}

\section{Classification of Shape Wilf Equivalent Classes for Quintuples of Length-3 Patterns}\label{SWE-5}
% \begin{table}[!htbp]
%     \centering
% \begin{tabular}{|l|l|l|l|l|l|}
% \noalign{\hrule height 1pt}
% \textbf{Class }&~~~~~~~~~~~~~~~~~~\textbf{Shape-Wilf-equivalent triples}  \\
% %none of them is in OEIS
% \noalign{\hrule height 1pt}

% I & $\{123,132,213,231,312\}\sim_s \{132,213,231,312,321\}$  \\

% \noalign{\hrule height 1pt}

% II & $\{123,132,213,231,321\}\sim_s \{123,132,213,312,321\}$   \\

% \noalign{\hrule height 1pt}

% III & $\{123,132,231,312,321\} \sim_s \{123,213,231,312,321\}$ \\

% \noalign{\hrule height 1pt}

% \end{tabular}
% \caption{The shape-Wilf-equivalent classes for quintuples of patterns of length $3$}
% \label{table3}
% \end{table}
In this section, we complete the classification of shape-Wilf-equivalence classes for quintuple pattern sets of length $3$. Unlike the quadruple case, the equivalences arising here can be established entirely through encoding arguments.

The enumerative aspects of these equivalence classes are investigated in Section \ref{matching} through the study of restricted perfect matchings. Since shape-Wilf-equivalent pattern sets yield identical avoidance counts, each quintuple equivalence class corresponds to a common matching enumeration sequence. For convenience, Table \ref{table3} records all equivalence classes together with their associated counting sequences, relevant references, and OEIS entries.

% \begin{table}[H]
%     \centering
%     \scriptsize{
%     {\renewcommand{\arraystretch}{1.5}
% \begin{tabular}{|l|l|l|l|l|l|}
% \noalign{\hrule height 1pt}
% \textbf{Class }&~~~~~~~~~~~\textbf{Shape-Wilf-equivalent quintuples $\mathbf{P}$}& $\mathbf{|\mathcal{M}_n(P)|~(n=1,2,\ldots,8,\ldots)}$ & \textbf{Ref and OEIS}\\
% %none of them is in OEIS
% \noalign{\hrule height 1pt}

% I & $\{123,132,213,231,312\}\sim_s \{132,213,231,312,321\}$ & $1, 3, 10, 35, 126, 462, 1716, 6435,\ldots$ &  Theorem \ref{catalan2},  ~\href{https://oeis.org/A001700}{A001700} \\

% \hline

% II & $\{123,132,213,231,321\}\sim_s \{123,132,213,312,321\}$ & $1, 3, 10, 33, 109, 360, 1189, 3927,\ldots$& Theorem \ref{linear_quintuple}, ~\href{https://oeis.org/A006190}{A006190}  \\

% \hline

% III & $\{123,132,231,312,321\} \sim_s \{123,213,231,312,321\}$ & $1, 3, 10, 34, 116, 396, 1352, 4616,\ldots$ &  Theorem \ref{linear_quintuple}, ~\href{https://oeis.org/A007052}{A007052}\\
 
% \noalign{\hrule height 1pt}

% \end{tabular}
% }}
% \caption{The shape-Wilf-equivalent classes for quintuples of patterns of length $3$}
% \label{table3}
% \end{table}

For readability, we label the quintuple pattern sets in Table \ref{table5} appearing in the classification. These labels will be used throughout the proofs and in the description of the shape-Wilf-equivalence classes.
\begin{table}[!htbp]
\centering
\begin{tabular}{c|c}
Notation & Pattern set\\
\hline
$Q_1$ & $\{123, 132, 213, 231, 312\}$\\
$Q_2$ & $\{123, 132, 213, 231, 321\}$\\
$Q_3$ & $\{123, 132, 213, 312, 321\}$\\
$Q_4$ & $\{123, 132, 231, 312, 321\}$\\
$Q_5$ & $\{123, 213, 231, 312, 321\}$\\
$Q_6$ & $\{132, 213, 231, 312, 321\}$\\

\end{tabular}
\caption{Quintuple pattern sets considered in this paper.}
\label{table5}
\end{table}

\begin{table}[H]
    \centering
    \scriptsize{
    {\renewcommand{\arraystretch}{1.5}
\begin{tabular}{|l|l|l|l|l|l|}
\noalign{\hrule height 1pt}
\textbf{Class }&\textbf{Shape-Wilf-equivalent quintuples $\mathbf{P}$}& $\mathbf{|\mathcal{M}_n(P)|~(n=1,2,\ldots,8,\ldots)}$ & \textbf{Ref and OEIS}\\
%none of them is in OEIS
\noalign{\hrule height 1pt}

\hyperref[quinI]{I} & $~~~~~~~~~~~~~~~~~~~~~~~Q_1\sim_s Q_6$ & $1, 3, 10, 35, 126, 462, 1716, 6435,\ldots$ &  Theorem \ref{catalan2},  ~\href{https://oeis.org/A001700}{A001700} \\

\hline

\hyperref[quinII]{II} &~~~~~~~~~~~~~~~~~~~~~~~$Q_2\sim_sQ_3$ & $1, 3, 10, 33, 109, 360, 1189, 3927,\ldots$& Theorem \ref{linear_quintuple}, ~\href{https://oeis.org/A006190}{A006190}  \\

\hline

\hyperref[quinIII]{III} & ~~~~~~~~~~~~~~~~~~~~~~~$Q_4 \sim_s Q_5$ & $1, 3, 10, 34, 116, 396, 1352, 4616,\ldots$ &  Theorem \ref{linear_quintuple}, ~\href{https://oeis.org/A007052}{A007052}\\
 
\noalign{\hrule height 1pt}

\end{tabular}
}}
\caption{The shape-Wilf-equivalent classes for quintuples of patterns of length $3$}
\label{table3}
\end{table}

\begin{note}
    %For $P_2=\{123,132,213,231,321\}$, $P_3=\{123,132,213,312,321\}$, $P_4=\{123,132,231,312,321\}$,\\ $P_5=\{123,213,231,312,321\}$
    If $P=Q_2,Q_3,Q_4,Q_5$ and the $i$-th row of a Ferrers board has more than three white cells, then the board admits no transversal that avoids $P$. If the $i$-th row of a Ferrers board has three white cells then the position for placing $1$ is fixed for all $Q_i$, $i=1,\ldots , 6$.
\end{note}
\paragraph{Encoding process.}
Let $Q\in\{Q_1,\ldots,Q_6\}$, and let $T$ be a $Q$-avoiding transversal of a Ferrers board $F$. We associate with $T$ a word $w=w_1\cdots w_n\in\{0,1,2\}^n.$
At the $i$-th step, consider the remaining white cells in the current top row. By the preceding lemmas, whenever the placement is not forced, this row contains exactly two white cells. We define
$$
w_i=
\begin{cases}
0, & \text{if the placement is forced},\\
1, & \text{if the leftmost of the two admissible cells is chosen},\\
2, & \text{if the rightmost of the two admissible cells is chosen}.
\end{cases}
$$
Thus, each $Q$-avoiding transversal determines a unique encoding word.

For each pair of pattern sets considered below, the corresponding lemmas show that the same stages of the construction are forced and unforced. Hence the two pattern sets admit exactly the same encoding words. Conversely, an encoding word determines the transversal uniquely: at a forced step, the symbol $0$ specifies the unique admissible position, while at an unforced step the symbols $1$ and $2$ specify the leftmost and rightmost admissible positions, respectively. Therefore, the common encoding gives a bijection between the corresponding families of avoiding transversals.

% \subsection{Class I (\texorpdfstring{$\mathbf{Q_1 \sim_s Q_6 }$}
% {Q1 ~s Q6})}
\subsection{Class I}\label{quinI}
In this section we will prove the shape-Wilf-equivalence between the pattern sets $Q_1,Q_6$.
% $\{123,132,213,231,312\} $ and $\{132,213,231,312,321\}$.

\begin{lemma}
At every stage of the process, the topmost row of white cells admits at least one admissible position for placing a $1$ that avoids any of the given pattern sets.
\end{lemma}

\begin{proof}
For the pattern set $Q_1=\{123,132,213,231,312\}$, placing the $1$ in the leftmost cell of the top row avoids all patterns in $P$ at every stage. Similarly, for $Q_6=\{132,213,231,312,321\}$, placing the $1$ in the rightmost cell of the top row is always admissible.  
In both cases, at least one valid placement exists.
\end{proof}

\begin{lemma}
    At any stage of the process, if the topmost row contains exactly two white cells, then either position may receive the next $1$. However, if the row contains at least three  white cells, then there is at most one admissible position for placing the next $1$ while avoiding the pattern set $P$.
    \begin{itemize}
        \item For $Q_1=\{123,132,213,231,312\}$, then the only valid position is the leftmost cell in that row.
        \item For $Q_6=\{132,213,231,312,321\}$, then the only valid position is the rightmost cell in that row.
    \end{itemize}
\end{lemma}

\begin{lemma}
Let $F$ be a Ferrers board, and let $\ell(c)$ denote the height of column $c$. At some stage of the construction for the pattern set, let $F'$ be the subboard of remaining white cells. Suppose the top row of $F'$ corresponds to row $r$ of $F$ and occupies columns $$c_1<\dots<c_l$$. Then the next placement is uniquely determined in any of the following situations: 

\begin{itemize}
    \item For $Q_1=\{123,132,213,231,312\}$ and the rectangle with vertices $(1,r+1)$, $(c_1-1,r+1)$, $(c_1-1,\ell(c_2))$, $(1,\ell(c_2))$ contains a $1$; then the only valid position is $c_1$.
    \item   For $Q_6=\{132,213,231,312,321\}$ and a $1$ occurs strictly above and to the right of $(c_l,r)$; then the only valid position is $c_l$.
    
\end{itemize}
\end{lemma}

% \subsection{Class II (\texorpdfstring{$\mathbf{Q_2 \sim_s Q_3 }$}
% {Q2 ~s Q3})}
\subsection{Class II}\label{quinII}
In this section we will prove the shape-Wilf-equivalence between the pattern sets  $Q_2,Q_3$.
% $\{123,132,213,231,321\} $ and $\{123,132,213,312,321\}$.

\begin{lemma}
    At any stage of the process, if the topmost row contains exactly two white cells, then either position may receive the next $1$. However, if the row contains three  white cells, then there is at most one admissible position for placing the next $1$ while avoiding the pattern set.
    \begin{itemize}
        \item For $Q_2=\{123,132,213,231,321\}$, then the only valid position is the leftmost cell in that row.
        \item for $Q_3=\{123,132,213,312,321\}$, then the only valid position is the middle cell in that row.
    \end{itemize}
\end{lemma}

\begin{lemma}\label{q2q3}
    In the Ferrers board $F$, if two consecutive row have three white cells then the board admits no transversal that avoids $Q_2, Q_3$.
\end{lemma}
\begin{proof}
    Let the $i$-th row of the Ferrers board contain exactly three white cells in columns
    $c_1 < c_2 < c_3$.

    First, consider $Q_2$. To avoid $Q_2$, the $1$ in the $i$-th row must be placed in column $c_1$. Suppose the $(i+1)$-st row contains white cells in columns $c_2, c_3, c_4$. Then avoiding $P$ forces the $1$ in that row to lie in column $c_2$. However, this configuration implies that the entries in columns $c_1, c_2, c_3$ form the pattern $321$, which belongs to $Q_2$, yielding a contradiction.

    Next, consider $Q_3$. In this case, avoidance of $Q_3$ forces the $1$ in the $i$-th row to be placed in column $c_2$. If the $(i+1)$-st row has white cells in columns $c_1, c_3, c_4$, then avoiding $P$ forces the $1$ in that row to be placed in column $c_3$. This placement results in an occurrence of the pattern $132$ in columns $c_1, c_2, c_3$, again contradicting the assumption that the transversal avoids $Q_3$.
\end{proof}
\begin{note}
    At each step, the placement of the $1$ is forced until reaching a row $i$ whose set of white cells consists of exactly two columns $c_1 < c_2$, with column $c_2$ having no cell above it in the Ferrers board.

\end{note}
\begin{lemma}
Let $F$ be a Ferrers board, and let $\ell(c)$ denote the height of column $c$. At some stage of the construction for $P$, let $F'$ be the subboard of the remaining white cells. Suppose that the top row of $F'$ corresponds to the row $r$ of $F$ and occupies columns $c_1<c_2$. Then the next placement is uniquely determined in any of the following situations:

\begin{itemize}
    \item For $Q_2=\{123,132,213,231,321\}$, if the rectangle with vertices $(1,r+1)$, $(c_1-1,r+1)$, $(c_1-1,\ell(c_2))$, $(1,\ell(c_2))$ contains a $1$; then the only valid position is $c_1$.
     \item  For $Q_3=\{123,132,213,312,321\}$, if a $1$ from stage $i-1$ lies in row $r+1$ between columns $c_1$ and $c_{2}$; then the only valid position is $c_1$.
\end{itemize}
\end{lemma}
% \subsection{Class III (\texorpdfstring{$\mathbf{Q_4 \sim_s Q_5 }$}
% {Q4 ~s Q5})}
\subsection{Class III}\label{quinIII}

In this section we will prove the shape-Wilf-equivalence between the pattern sets $Q_4,Q_5$. 
% $\{123,132,231,312,321\} $ and $ \{123,213,231,312,321\}$.
\begin{lemma}
    At any stage of the process, if the topmost row contains exactly two white cells, then either position may receive the next $1$. However, if the row contains three  white cells, then there is at most one admissible position for placing the next $1$ while avoiding the pattern set $P$.
    \begin{itemize}
        \item For $Q_4=\{123,132,231,312,321\}$, then the only valid position is the rightmost cell in that row.
        \item For $Q_5=\{123,213,231,312,321\}$, then the only valid position is the middle cell in that row.
    \end{itemize}
\end{lemma}

\begin{lemma}
Let $F$ be a Ferrers board, and let $\ell(c)$ denote the height of column $c$. At some stage of the construction for $P$, let $F'$ be the subboard of remaining white cells. Suppose the top row of $F'$ corresponds to row $r$ of $F$ and occupies columns $c_1<\ldots<c_l$, where $l=2$ or $3$. Then the next placement is uniquely determined in any of the following situations:

\begin{itemize}
   \item    $Q_4=\{123,132,231,312,321\}$ and a $1$ occurs strictly above and to the right of $(c_3,r)$; then the only valid position is $c_1$.
    \item  $Q_5=\{123,213,231,312,321\}$ and a $1$ from stage $i-1$ lies in row $r+1$ between columns $c_k$ and $c_{k+1}$; then the only valid position is $c_{k+1}$. 
\end{itemize}
\end{lemma}

\begin{note}
\begin{itemize}
    \item  Although the lemmas and the encoding process for Class III are analogous to those for Class II, Lemma \ref{q2q3} does not hold in this case.
    \item  The remaining two singleton classes IV and V were identified through computational experiments using Python.
\end{itemize}
   
\end{note}

\section{Patterns in Matchings}\label{matching}
We use the same notation and terminology as in Part I \cite{SUK-arxiv}. For the convenience of the reader, we briefly recall the necessary definitions and conventions. Let $\mathcal{F}_n$ be the set of all Ferrers boards with $n$ rows and $n$ columns 
containing the staircase board. Let $\mathcal{T}_F$ be the set of all transversals of 
the Ferrers board $F$ and 
%Let $\mathcal{T}_F$ denote the set of all transversals of the Ferrers board $F$ and
let $\mathcal{T}_F(P)$ denote the set of all transversals of the Ferrers board $F$ avoiding the set of patterns $P$. Let
$$\mathcal{T}_n=\bigcup_{F \in \mathcal{F}_n}\mathcal{T}_F, ~~~~~~~~~
\mathcal{T}_n(P)=\bigcup_{F \in \mathcal{F}_n}\mathcal{T}_F(P).$$ and 
let $\mathcal{M}_n$ be the set of perfect matchings on $[2n]$. There exists a natural bijection between $\mathcal{M}_n$ and $\mathcal{T}_n$ (see Bloom and Elizalde \cite{matching_partition}). For a fixed $F \in \mathcal{F}_n$, the set of matchings of shape $F$ is denoted by $\mathcal{M}_F$ and by $\mathcal{M}_F(P)$ those that avoid the set of patterns $P$. They proved that (Section $2.2$)  $|\mathcal{M}_F(p)|=|\mathcal{T}_F(p)|$ for all pattern $p$, which implies $|\mathcal{M}_F(P)|=|\mathcal{T}_F(P)|$ for all set of patterns $P$. Thus, shape-Wilf equivalence can be interpreted in terms of pattern-avoiding matchings: $P_1 \sim_s P_2$ if and only if  $|\mathcal{M}_F(P_1)|=|\mathcal{M}_F(P_2)|$ for every Ferrers board $F$.

Let {\small $A_n(P)=\{(T,F): F \in \mathcal{F}_n \mbox{  and }T \mbox{ is a transversal of }F\mbox{ avoiding the set of patterns }P \}$ and $a_n=|A_n(P)|$}. Therefore $|\mathcal{M}_n(P)|=a_n$. Note that if $P_1\sim_s P_2$ then $|A_n(P_1)|=|A_n(P_2)|.$

\begin{figure}[H]

   \centering
{
\begin{minipage}{.3\textwidth}

\begin{tikzpicture}[scale=.6]
	
        \filldraw  (0,0) circle (2pt); 
        \filldraw  (1,0) circle (2pt); 
        \filldraw  (2,0) circle (2pt); 
        \filldraw  (6,0) circle (2pt); 
        \filldraw  (5,0) circle (2pt); 
        \filldraw  (4,0) circle (2pt);

        \draw (0,0) arc (180:0:2) ;
        \draw (1,0) arc (180:0:2) ;
        \draw (2,0) arc (180:0:2) ;
        %\draw[dotted][thick] (4,0) arc (180:0:1.5) ;
        % \draw (0,0) arc (0:180:4) ;
        % \draw (0,0) arc (-0:0:2) ;
        \node[] at (0,-0.4)   {  \tiny $1$};
        \node[] at (1,-0.4)   {  \tiny $2$};
        \node[] at (2,-0.4)   {  \tiny $3$};
        \node[] at (4,-0.4)   { \tiny {$4$}};
        \node[] at (5,-0.4)   {  \tiny{$5$}};
         \node[] at(6,-0.4)   {\tiny $6$};
         \node[] at(3,-1.25)   { $321$};
        % \foreach \start/\stop in {50/130}
        % {
        % \draw[red] (10,.50) arc (\start:\stop:4);
        % }
        
        % \node[] at (14.5,-5.5)   { $(0)$};
        % \node[] at (-4.5,5.5)   { $(0)$};
\end{tikzpicture}
\end{minipage}
\begin{minipage}{.3\textwidth}
  \begin{tikzpicture}[scale=.6]
       \filldraw  (0,0) circle (2pt); 
        \filldraw  (1,0) circle (2pt); 
        \filldraw  (2,0) circle (2pt); 
        \filldraw  (3,0) circle (2pt); 
        \filldraw  (5,0) circle (2pt); 
        \filldraw  (4,0) circle (2pt);

        \draw (0,0) arc (180:0:2.5) ;
        \draw (1,0) arc (180:0:1.5) ;
        \draw (2,0) arc (180:0:.5) ;
        
        \node[] at (0,-0.4)   {  \tiny $1$};
        \node[] at (1,-0.4)   {  \tiny $2$};
        \node[] at (2,-0.4)   {  \tiny $3$};
        \node[] at(3,-0.4)   {\tiny $4$};
        \node[] at (4,-0.4)   { \tiny {$5$}};
        \node[] at (5,-0.4)   {  \tiny{$6$}};
         
         \node[] at(2.5,-1.25)   { $123$};
  \end{tikzpicture}  
\end{minipage}
\begin{minipage}{.3\textwidth}
  \begin{tikzpicture}[scale=.6]
       \filldraw  (0,0) circle (2pt); 
        \filldraw  (.75,0) circle (2pt); 
        \filldraw  (1.75,0) circle (2pt); 
        \filldraw  (3.25,0) circle (2pt); 
        \filldraw  (5,0) circle (2pt); 
        \filldraw  (4.25,0) circle (2pt);

        \draw (0,0) arc (180:0:2.5) ;
        \draw (.75,0) arc (180:0:1.25) ;
        \draw (1.75,0) arc (180:0:1.25) ;
        
        \node[] at (0,-0.4)   {  \tiny $1$};
        \node[] at (.75,-0.4)   {  \tiny $2$};
        \node[] at (1.75,-0.4)   {  \tiny $3$};
        \node[] at(3.25,-0.4)   {\tiny $4$};
        \node[] at (4.25,-0.4)   { \tiny {$5$}};
        \node[] at (5,-0.4)   {  \tiny{$6$}};
         
         \node[] at(2.5,-1.25)   { $132$};
  \end{tikzpicture}  
\end{minipage}

\begin{minipage}{.3\textwidth}

\begin{tikzpicture}[scale=.6]
	
        \filldraw  (0,0) circle (2pt); 
        \filldraw  (1,0) circle (2pt); 
        \filldraw  (2.5,0) circle (2pt); 
        \filldraw  (6.5,0) circle (2pt); 
        \filldraw  (5,0) circle (2pt); 
        \filldraw  (4,0) circle (2pt);

        \draw (0,0) arc (180:0:2.5) ;
        \draw (1,0) arc (180:0:1.5) ;
        \draw (2.5,0) arc (180:0:2) ;
        %\draw[dotted][thick] (4,0) arc (180:0:1.5) ;
        % \draw (0,0) arc (0:180:4) ;
        % \draw (0,0) arc (-0:0:2) ;
        \node[] at (0,-0.4)   {  \tiny $1$};
        \node[] at (1,-0.4)   {  \tiny $2$};
        \node[] at (2.5,-0.4)   {  \tiny $3$};
        \node[] at (4,-0.4)   { \tiny {$4$}};
        \node[] at (5,-0.4)   {  \tiny{$5$}};
         \node[] at(6.5,-0.4)   {\tiny $6$};
         \node[] at(3.5,-1.25)   { $231$};
       
\end{tikzpicture}
\end{minipage}
\begin{minipage}{.3\textwidth}
\begin{tikzpicture}[scale=.55]

        \filldraw  (0,0) circle (2pt); 
        \filldraw  (1.5,0) circle (2pt); 
        \filldraw  (2.5,0) circle (2pt); 
        \filldraw  (6.5,0) circle (2pt); 
        \filldraw  (5.5,0) circle (2pt); 
        \filldraw  (4,0) circle (2pt);

        \draw (1.5,0) arc (180:0:2.5) ;
        \draw (2.5,0) arc (180:0:1.5) ;
        \draw (0,0) arc (180:0:2) ;
        %\draw[dotted][thick] (4,0) arc (180:0:1.5) ;
        % \draw (0,0) arc (0:180:4) ;
        % \draw (0,0) arc (-0:0:2) ;
        \node[] at (0,-0.4)   {  \tiny $1$};
        \node[] at (1.5,-0.4)   {  \tiny $2$};
        \node[] at (2.5,-0.4)   {  \tiny $3$};
        \node[] at (4,-0.4)   { \tiny {$4$}};
        \node[] at (5.5,-0.4)   {  \tiny{$5$}};
         \node[] at(6.5,-0.4)   {\tiny $6$};
         \node[] at(3.5,-1.25)   { $312$};
       
\end{tikzpicture}
\end{minipage}
\begin{minipage}{.3\textwidth}
    \begin{tikzpicture}[scale=.55]

        \filldraw  (0,0) circle (2pt); 
        \filldraw  (1.25,0) circle (2pt); 
        \filldraw  (2.15,0) circle (2pt); 
        \filldraw  (6.25,0) circle (2pt); 
        \filldraw  (5,0) circle (2pt); 
        \filldraw  (4.15,0) circle (2pt);

        \draw (0,0) arc (180:0:2.5) ;
        \draw (1.25,0) arc (180:0:2.5) ;
        \draw (2.15,0) arc (180:0:1) ;
        %\draw[dotted][thick] (4,0) arc (180:0:1.5) ;
        % \draw (0,0) arc (0:180:4) ;
        % \draw (0,0) arc (-0:0:2) ;
        \node[] at (0,-0.4)   {  \tiny $1$};
        \node[] at (1.25,-0.4)   {  \tiny $2$};
        \node[] at (2.15,-0.4)   {  \tiny $3$};
        \node[] at (4.15,-0.4)   { \tiny {$4$}};
        \node[] at (5,-0.4)   {  \tiny{$5$}};
         \node[] at(6.25,-0.4)   {\tiny $6$};
         \node[] at(3.25,-1.25)   { $213$};
       
\end{tikzpicture}
\end{minipage}
}
\caption{Patterns in matchings corresponding to permutations of length $3$}  
\label{Fig1}
 \end{figure}
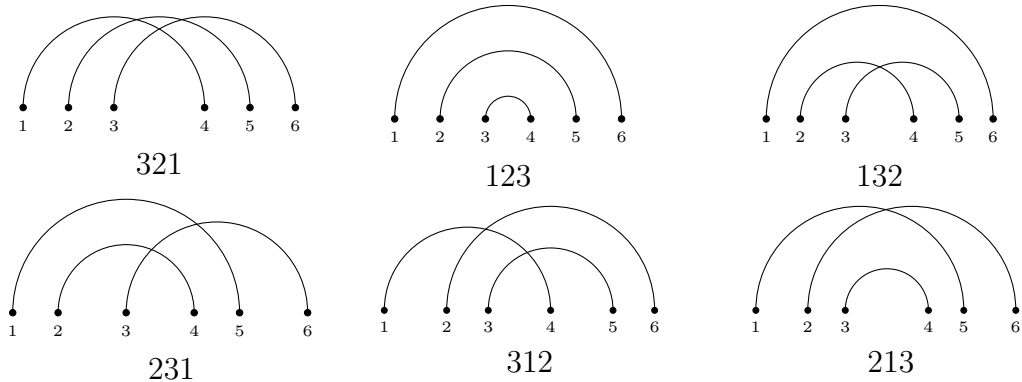 

% \begin{theorem}{(\cite{enumeration_matching}, Corollary 3.2)}
%     Let $P=\{132,231,312\}$, then the generating function of $a_n$ is $\frac{1}{1-xC(x)C(C(x)-1)}$ where $C(x)$ is the Catalan polynomial.
% \end{theorem}

\begin{definition}
  A perfect matching $M$ on $[2n]$ is said to be \emph{disconnected} if there exists an integer $k$, with $1\leq k<n$, such that no arc of $M$ has one endpoint in $[2k]=\{1,2,\ldots,2k\}$
and the other endpoint in $\{2k+1,2k+2,\ldots,2n\}.$
Equivalently, $M$ can be written as
$$M=M'\oplus M'',$$
where $M'$ and $M''$ are nonempty perfect matchings on two consecutive intervals of vertices. Otherwise, $M$ is said to be \emph{connected}.
\end{definition}

\begin{note}
    As established in the previous sections, there are five shape-Wilf-equivalence classes for quadruple pattern sets and three shape-Wilf-equivalence classes for quintuple pattern sets. A result of Bloom and Elizalde \cite{matching_partition} implies that shape-Wilf-equivalent pattern sets have the same number of avoiding perfect matchings. Consequently, in order to enumerate matchings avoiding pattern sets from a given equivalence class, it suffices to study a single representative from that class. Therefore, for each shape-Wilf-equivalence class, we select one representative pattern set and determine the corresponding enumeration sequence.
\end{note}

\begin{theorem}\label{linear_quadruple1}
     For $P_{14}=\{123,213,312,321\}$,   $a_{n}=4a_{n-1}-a_{n-2}$ for all $n\geq 3$ and $a_0=1$, $a_1=1$, $a_2=3$.
\end{theorem}
\begin{proof}
 We prove the result by induction on $n$. Given a matching on $2(n-1)$ or $2(n-2)$ vertices,
we obtain a matching on $2n$ vertices by adding  a single new rightmost opening arc or a pair of rightmost opening arcs. Since the matchings are required to avoid the patterns $123$, $213$, $312$ and $321$, it follows from Figure \ref{Fig1} that they cannot contain three nested arcs (corresponding to $123$), three crossing arcs (corresponding to $321$), a nesting under a crossing (corresponding to $213$),  two nestings together with  a left crossing (corresponding to  $312$). 
Consequently, to ensure that the resulting matching avoids the pattern set $P_{14}$, only certain placements of the final arc (or the final two arcs) are allowed.
To add a single arc or a pair of arcs, we relabel the $2(n-1)$ or $2(n-2)$ vertices accordingly so that the new rightmost arc(s) can be properly positioned on $2n$ vertices. These admissible configurations are illustrated in Figure \ref{Fig2}. For simplicity, in Figure \ref{Fig2} we denote the vertex $2n-i$ by $-i$, for $1\le i\le 2n-1$.
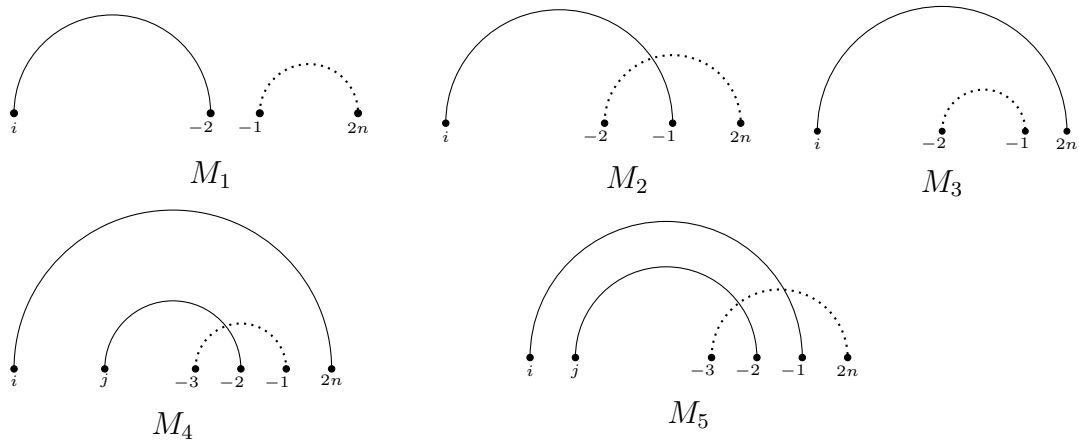
\begin{figure}[H]
    \begin{minipage}{.35\textwidth}
    \begin{tikzpicture}[scale=.65]

        \filldraw  (0,0) circle (2pt); 
        \filldraw  (7,0) circle (2pt); 
        \filldraw  (5,0) circle (2pt); 
        \filldraw  (4,0) circle (2pt); 
        \draw (0,0) arc (180:0:2) ;
        \draw[dotted][thick] (5,0) arc (180:0:1) ;
        % \draw (0,0) arc (0:180:4) ;
        % \draw (0,0) arc (-0:0:2) ;
        \node[] at (0,-0.3)   { \tiny $i$};
        \node[] at (3.8,-0.3)   {\tiny $-2$};
         \node[] at (4.8,-0.3)   {\tiny $-1$};
         \node[] at (7,-0.3)   { \tiny$2n$};
         \node[] at (4,-1.25)   { $M_1$};

\end{tikzpicture}
    \end{minipage}
     \begin{minipage}{.3\textwidth}
     \begin{tikzpicture}[scale=.6]
	
        \filldraw  (0,0) circle (2pt); 
        \filldraw  (6.5,0) circle (2pt); 
        \filldraw  (5,0) circle (2pt); 
        \filldraw  (3.5,0) circle (2pt); 
        
        \draw (0,0) arc (180:0:2.5) ;
        \draw[dotted][thick] (3.5,0) arc (180:0:1.5) ;
        % \draw (0,0) arc (0:180:4) ;
        % \draw (0,0) arc (-0:0:2) ;
        \node[] at (0,-0.3)   { \tiny$i$};
        \node[] at (3.3,-0.3)   { \tiny $-2$};
         \node[] at (4.8,-0.3)   { \tiny $-1$};
         \node[] at (6.5,-0.3)   { \tiny $2n$};
        
         \node[] at (4,-1.25)   { $M_2$};
\end{tikzpicture}

     \end{minipage}
\begin{minipage}{.3\textwidth}
     \begin{tikzpicture}[scale=.55]

      \filldraw  (0,0) circle (2pt); 
        \filldraw  (6,0) circle (2pt); 
        \filldraw  (5,0) circle (2pt); 
        \filldraw  (3,0) circle (2pt); 
        \draw (0,0) arc (180:0:3) ;
        \draw[dotted][thick] (3,0) arc (180:0:1) ;
        % \draw (0,0) arc (0:180:4) ;
        % \draw (0,0) arc (-0:0:2) ;
        \node[] at (0,-0.3)   { \tiny$i$};
        \node[] at (2.8,-0.3)   { \tiny$-2$};
         \node[] at (4.8,-0.3)   { \tiny$-1$};
         \node[] at (6,-0.3)   { \tiny$2n$};
         \node[] at (3,-1.25)   { $M_3$};
        
     \end{tikzpicture}
\end{minipage}

\begin{minipage}{.42\textwidth}
     \begin{tikzpicture}[scale=.6]
      \filldraw  (0,0) circle (2pt); 
        \filldraw  (6,0) circle (2pt); 
        \filldraw  (5,0) circle (2pt); 
        \filldraw  (4,0) circle (2pt); 
        \filldraw  (7,0) circle (2pt);
        \filldraw  (2,0) circle (2pt); 
        \draw (0,0) arc (180:0:3.5) ;
        \draw (2,0) arc (180:0:1.5) ;
        \draw[dotted][thick] (4,0) arc (180:0:1) ;
        % \draw (0,0) arc (0:180:4) ;
        % \draw (0,0) arc (-0:0:2) ;
        \node[] at (0,-0.25)   { \tiny$i$};
        \node[] at (2,-0.25)   { \tiny$j$};
        \node[] at (3.8,-0.3)   { \tiny$-3$};
         \node[] at (4.8,-0.3)   { \tiny$-2$};
         \node[] at (5.8,-0.3)   { \tiny$-1$};
         \node[] at (7,-0.25)   { \tiny$2n$};

         \node[] at (3.5,-1.25)   { $M_4$};
     \end{tikzpicture}
\end{minipage}
\begin{minipage}{.3\textwidth}
     \begin{tikzpicture}[scale=.6]
      \filldraw  (0,0) circle (2pt); 
        \filldraw  (1,0) circle (2pt); 
        \filldraw  (7,0) circle (2pt); 
        \filldraw  (6,0) circle (2pt); 
        \filldraw  (5,0) circle (2pt); 
        \filldraw  (4,0) circle (2pt);

        \draw (0,0) arc (180:0:3) ;
        \draw (1,0) arc (180:0:2) ;
        \draw[dotted][thick] (4,0) arc (180:0:1.5) ;
        % \draw (0,0) arc (0:180:4) ;
        % \draw (0,0) arc (-0:0:2) ;
        \node[] at (0,-0.3)   {  \tiny $i$};
        \node[] at (1,-0.3)   {  \tiny $j$};
        \node[] at (3.8,-0.3)   { \tiny {$-3$}};
        \node[] at (4.8,-0.3)   {  \tiny{$-2$}};
         \node[] at (5.8,-0.3)   {\tiny $-1$};
         \node[] at (7,-0.3)   { \tiny $2n$};
         \node[] at (3.5,-1.25)   { $M_5$};
     \end{tikzpicture}
\end{minipage}

\caption{Permitted matchings}  
\label{Fig2}
 \end{figure}
Let $m_i(n)$ denote the number of matchings of type $M_i$ on $2n$ vertices,
for $i=1,2,\ldots,5$, and let $a_n$ be the total number of $P_{14}$-avoiding matchings
on $2n$ vertices.

Matchings of types $M_1$, $M_2$, and $M_3$ can be obtained by adding a single arc
to any $P_{14}$-avoiding matching on $2(n-1)$ vertices, without imposing any additional
conditions. Hence
$$m_1(n)=m_2(n)=m_3(n)=a_{n-1}.$$

Moreover, $m_4(n)=m_3(n-1)=a_{n-2}.$
The matchings that end with a single arc are precisely those of types $M_1$ and $M_2$. Hence they cannot contribute to $M_5$. Consequently,
$$m_5(n)=a_{n-1}-m_1(n-2)-m_2(n-1)=a_{n-1}-2a_{n-2}.$$ 
Therefore, $$ a_n=\sum_{i=1}^{5} m_i(n)=4a_{n-1}-a_{n-2},$$
which completes the proof.
\end{proof}

\begin{theorem}\label{linear_quadruple3}
    For $P_{13}=\{123,231,312,321\}$,  $a_n=6a_{n-1}-9a_{n-2}+3a_{n-3}$ for $n\geq 4$ and  $a_1=1$, $a_2=3$, $a_3=11$.
\end{theorem}
\begin{proof}
We prove the result by induction on $n$. Given a matching on $2(n-1)$ vertices,
we obtain a matching on $2n$ vertices by adding a single new rightmost opening arc. Since the matchings are required to avoid the patterns $123$, $231$, $312$ and $321$, it follows from Figure \ref{Fig1} that they cannot contain three nested arcs (corresponding to $123$), three crossing arcs (corresponding to $321$), nor can it contain two nestings together with either a left or a right crossing (corresponding to $231$ and $312$).
Consequently, to ensure that the resulting matching avoids the pattern set $P_{13}$, only certain placements of the final arc are allowed.
To add   a single arc, we relabel the $2(n-1)$ vertices accordingly so that the new rightmost arc can be properly positioned on $2n$ vertices. These admissible configurations are illustrated in Figure \ref{Fig3}. For simplicity, in Figure \ref{Fig3} we denote the vertex $2n-i$ by $-i$, for $1\le i\le 2n-1$.

    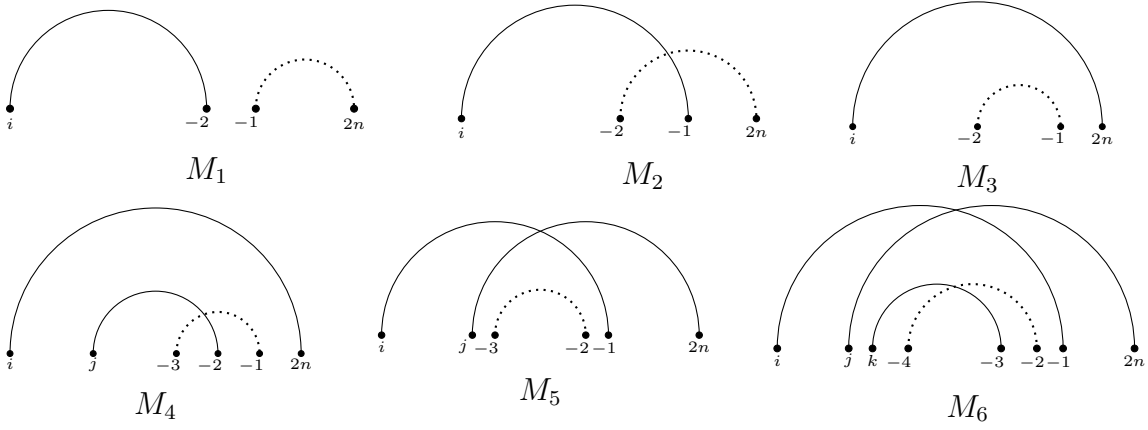
\begin{figure}[H]
    \begin{minipage}{.35\textwidth}
    \begin{tikzpicture}[scale=.65]

        \filldraw  (0,0) circle (2pt); 
        \filldraw  (7,0) circle (2pt); 
        \filldraw  (5,0) circle (2pt); 
        \filldraw  (4,0) circle (2pt); 
        \draw (0,0) arc (180:0:2) ;
        \draw[dotted][thick] (5,0) arc (180:0:1) ;
        % \draw (0,0) arc (0:180:4) ;
        % \draw (0,0) arc (-0:0:2) ;
        \node[] at (0,-0.3)   { \tiny $i$};
        \node[] at (3.8,-0.3)   {\tiny $-2$};
         \node[] at (4.8,-0.3)   {\tiny $-1$};
         \node[] at (7,-0.3)   { \tiny$2n$};
         \node[] at (4,-1.25)   { $M_1$};

\end{tikzpicture}
    \end{minipage}
     \begin{minipage}{.3\textwidth}
     \begin{tikzpicture}[scale=.6]
	
        \filldraw  (0,0) circle (2pt); 
        \filldraw  (6.5,0) circle (2pt); 
        \filldraw  (5,0) circle (2pt); 
        \filldraw  (3.5,0) circle (2pt); 
        
        \draw (0,0) arc (180:0:2.5) ;
        \draw[dotted][thick] (3.5,0) arc (180:0:1.5) ;
        % \draw (0,0) arc (0:180:4) ;
        % \draw (0,0) arc (-0:0:2) ;
        \node[] at (0,-0.3)   { \tiny$i$};
        \node[] at (3.3,-0.3)   { \tiny $-2$};
         \node[] at (4.8,-0.3)   { \tiny $-1$};
         \node[] at (6.5,-0.3)   { \tiny $2n$};
        
         \node[] at (4,-1.25)   { $M_2$};
\end{tikzpicture}
\end{minipage}
\begin{minipage}{.3\textwidth}
     \begin{tikzpicture}[scale=.55]

      \filldraw  (0,0) circle (2pt); 
        \filldraw  (6,0) circle (2pt); 
        \filldraw  (5,0) circle (2pt); 
        \filldraw  (3,0) circle (2pt); 
        \draw (0,0) arc (180:0:3) ;
        \draw[dotted][thick] (3,0) arc (180:0:1) ;
        % \draw (0,0) arc (0:180:4) ;
        % \draw (0,0) arc (-0:0:2) ;
        \node[] at (0,-0.3)   { \tiny$i$};
        \node[] at (2.8,-0.3)   { \tiny$-2$};
         \node[] at (4.8,-0.3)   { \tiny$-1$};
         \node[] at (6,-0.3)   { \tiny$2n$};
         \node[] at (3,-1.25)   { $M_3$};
        
     \end{tikzpicture}
\end{minipage}
\begin{minipage}{.3\textwidth}
     \begin{tikzpicture}[scale=.55]
      \filldraw  (0,0) circle (2pt); 
        \filldraw  (6,0) circle (2pt); 
        \filldraw  (5,0) circle (2pt); 
        \filldraw  (4,0) circle (2pt); 
        \filldraw  (7,0) circle (2pt);
        \filldraw  (2,0) circle (2pt); 
        \draw (0,0) arc (180:0:3.5) ;
        \draw (2,0) arc (180:0:1.5) ;
        \draw[dotted][thick] (4,0) arc (180:0:1) ;
        % \draw (0,0) arc (0:180:4) ;
        % \draw (0,0) arc (-0:0:2) ;
        \node[] at (0,-0.25)   { \tiny$i$};
        \node[] at (2,-0.25)   { \tiny$j$};
        \node[] at (3.8,-0.3)   { \tiny$-3$};
         \node[] at (4.8,-0.3)   { \tiny$-2$};
         \node[] at (5.8,-0.3)   { \tiny$-1$};
         \node[] at (7,-0.25)   { \tiny$2n$};

         \node[] at (3.5,-1.25)   { $M_4$};
     \end{tikzpicture}
\end{minipage}
\begin{minipage}{.32\textwidth}
     \begin{tikzpicture}[scale=.6]
       \filldraw  (0,0) circle (2pt); 
        \filldraw  (4.5,0) circle (2pt); 
        \filldraw  (5,0) circle (2pt); 
        \filldraw  (2.5,0) circle (2pt); 
        \filldraw  (7,0) circle (2pt);
        \filldraw  (2,0) circle (2pt); 
        \draw (0,0) arc (180:0:2.5) ;
        \draw (2,0) arc (180:0:2.5) ;
        \draw[dotted][thick] (2.5,0) arc (180:0:1) ;
        % \draw (0,0) arc (0:180:4) ;
        % \draw (0,0) arc (-0:0:2) ;
        \node[] at (0,-0.25)   { \tiny$i$};
        \node[] at (1.8,-0.25)   { \tiny$j$};
        \node[] at (2.3,-0.3)   { \tiny$-3$};
         \node[] at (4.3,-0.3)   { \tiny$-2$};
         \node[] at (4.9,-0.3)   { \tiny$-1$};
         \node[] at (7,-0.25)   { \tiny$2n$};

         \node[] at (3.5,-1.25)   { $M_5$};
     \end{tikzpicture}
\end{minipage}
\begin{minipage}{.35\textwidth}
\begin{tikzpicture}[scale=.63]
\filldraw  (0,0) circle (2pt); 
        \filldraw  (2.75,0) circle (2pt); 
        \filldraw  (4.7,0) circle (2pt); 
        \filldraw  (5.45,0) circle (2pt); 
        \filldraw  (6,0) circle (2pt); 
        \filldraw  (2,0) circle (2pt); 
        \filldraw  (7.5,0) circle (2pt);
        \filldraw  (1.5,0) circle (2pt); 
        \draw (0,0) arc (180:0:3) ;
        \draw (1.5,0) arc (180:0:3) ;
        \draw (2,0) arc (180:0:1.35) ;
        \draw[dotted][thick] (2.75,0) arc (180:0:1.35) ;
        % \draw (0,0) arc (0:180:4) ;
        % \draw (0,0) arc (-0:0:2) ;
        \node[] at (0,-0.25)   { \tiny$i$};
        \node[] at (1.5,-0.25)   { \tiny$j$};
        \node[] at (2,-0.25)   { \tiny$k$};
        \node[] at (2.55,-0.3)   { \tiny$-4$};
        \node[] at (4.5,-0.3)   { \tiny$-3$};
         \node[] at (5.35,-0.3) { \tiny$-2$};
         \node[] at (5.9,-0.3)   { \tiny$-1$};
         \node[] at (7.5,-0.25) { \tiny$2n$};

         \node[] at (4,-1.25)   { $M_6$};
\end{tikzpicture}
\end{minipage}
\caption{Permitted matchings}  
\label{Fig3}
 \end{figure}
Let $m_i(n)$ denote the number of matchings of type $M_i$ on $2n$ vertices,
for $i=1,2,\ldots,6$, and let $a_n$ be the total number of $P_{13}$-avoiding matchings
on $2n$ vertices.
 
Matchings of types $M_1$, $M_2$, and $M_3$ can be obtained by adding a single arc
to any $P_{13}$-avoiding matching on $2(n-1)$ vertices, without imposing any additional
conditions. Hence,
$$
m_1(n)=m_2(n)=m_3(n)=a_{n-1}.
$$

Matchings of type $M_4$ on $2n$ vertices can be constructed from matchings of type $M_3$ and $M_4$ on $2(n-1)$ vertices by adding a rightmost opening arc nested under the rightmost closer and crossing the second rightmost closer. Thus,
$$
m_4(n)=m_3(n-1)+m_4(n-1)=a_{n-2}+m_4(n-1).
$$
Note that $m_4(3)=1=a_1$, and hence $m_4(4)=a_1+a_2$. By induction, we obtain
\begin{equation}\label{eq1}
    m_4(n)= a_{n-2}+ \sum_{i=1}^{n-3} a_i
    = \sum_{i=1}^{n-2} a_i.
\end{equation}

Matchings of type $M_5$ on $2n$ vertices can be constructed from matchings of type $M_2$, $M_5$, and $M_6$ on $2(n-1)$ vertices by adding a rightmost opening arc nested under a crossing. Hence,
$$
m_5(n)=m_2(n-1)+m_5(n-1)+m_6(n-1)=a_{n-2}+m_5(n-1)+m_6(n-1).
$$
Since matchings of types $M_1$, $M_3$, and $M_4$ cannot contribute to $M_5$, we obtain
$$
m_5(n)=a_{n-1}-m_1(n-1)-m_3(n-1)-m_4(n-1),
$$
which simplifies to
\begin{equation}\label{eq2}
   m_5(n) = a_{n-1}-2a_{n-2}-m_4(n-1)
   = a_{n-1}-a_{n-2}-m_4(n).
\end{equation}

Matchings of type $M_6$ on $2n$ vertices can be constructed from matchings of type $M_5$ and $M_6$ on $2(n-1)$ vertices by adding a rightmost opening arc nested under a crossing and crossing the second rightmost closer. Therefore,
\begin{equation}\label{eq3}
    m_6(n)=m_5(n-1)+m_6(n-1)
    = m_5(n)-a_{n-2}
    = a_{n-1}-2a_{n-2}-m_4(n).
\end{equation}

Summing over all types and using equations \eqref{eq2} and \eqref{eq3}, we obtain
\begin{equation}\label{eq4}
    a_n=\sum_{i=1}^{6} m_i(n)=5a_{n-1}-3a_{n-2}-m_4(n).
\end{equation}

Combining equations \eqref{eq1} and \eqref{eq4}, we obtain
$$a_n=6a_{n-1}-9a_{n-2}+3a_{n-3}, \quad \text{for } n\geq 4,$$

which completes the proof.
 
\end{proof}

\begin{theorem}\label{linear_quadruple2}
    For $P_3=\{123,132,213,321\}$,  $a_{n}=3a_{n-1}+4a_{n-2}-3a_{n-3}$ for all $n\geq 4$ and $a_1=1$, $a_2=3$, $a_3=11$, $a_4=42$.
\end{theorem}
\begin{proof}
    We prove the result by induction on $n$. Given a matching on $2(n-1)$ or $2(n-2)$ vertices, we obtain a matching on $2n$ vertices by adding  a single new rightmost opening arc or a pair of rightmost opening arcs. Since the matchings are required to avoid the patterns $123$, $132$, $213$, and $321$, it follows from Figure \ref{Fig1} that they cannot contain three nested arcs (corresponding to $123$), three crossing arcs (corresponding to $321$), a crossing under a nesting (corresponding to $132$), a nesting under a crossing (corresponding to $213$). 
Consequently, to ensure that the resulting matching avoids the pattern set $P_3$, only certain placements of the final arc (or the final two arcs) are allowed.
To add a single arc or a pair of arcs, we relabel the $2(n-1)$ or $2(n-2)$ vertices accordingly so that the new rightmost arc(s) can be properly positioned on $2n$ vertices. These admissible configurations are illustrated in Figure \ref{Fig7}. For simplicity, in Figure \ref{Fig7} we denote the vertex $2n-i$ by $-i$, for $1\le i\le 2n-1$.
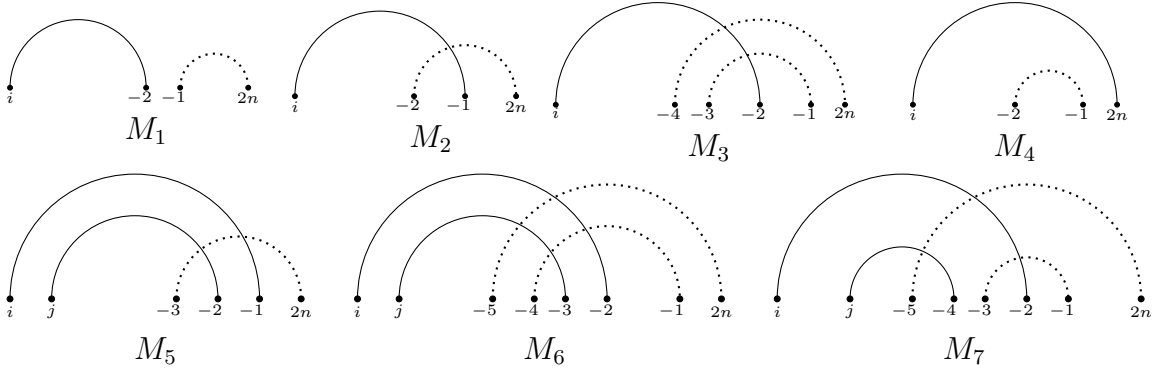
\begin{figure}[H]
    \begin{minipage}{.23\textwidth}
    \begin{tikzpicture}[scale=.45]

        \filldraw  (0,0) circle (2pt); 
        \filldraw  (7,0) circle (2pt); 
        \filldraw  (5,0) circle (2pt); 
        \filldraw  (4,0) circle (2pt); 
        \draw (0,0) arc (180:0:2) ;
        \draw[dotted][thick] (5,0) arc (180:0:1) ;
        % \draw (0,0) arc (0:180:4) ;
        % \draw (0,0) arc (-0:0:2) ;
        \node[] at (0,-0.3)   { \tiny $i$};
        \node[] at (3.8,-0.3)   {\tiny $-2$};
         \node[] at (4.8,-0.3)   {\tiny $-1$};
         \node[] at (7,-0.3)   { \tiny$2n$};
         \node[] at (4,-1.25)   { $M_1$};

\end{tikzpicture}
    \end{minipage}
     \begin{minipage}{.21\textwidth}
     \begin{tikzpicture}[scale=.45]
	
        \filldraw  (0,0) circle (2pt); 
        \filldraw  (6.5,0) circle (2pt); 
        \filldraw  (5,0) circle (2pt); 
        \filldraw  (3.5,0) circle (2pt); 
        
        \draw (0,0) arc (180:0:2.5) ;
        \draw[dotted][thick] (3.5,0) arc (180:0:1.5) ;
        % \draw (0,0) arc (0:180:4) ;
        % \draw (0,0) arc (-0:0:2) ;
        \node[] at (0,-0.3)   { \tiny$i$};
        \node[] at (3.3,-0.3)   { \tiny $-2$};
         \node[] at (4.8,-0.3)   { \tiny $-1$};
         \node[] at (6.5,-0.3)   { \tiny $2n$};
        
         \node[] at (4,-1.25)   { $M_2$};
\end{tikzpicture}

     \end{minipage}
\begin{minipage}{.29\textwidth}
     \begin{tikzpicture}[scale=.45]
      \filldraw  (-0.5,0) circle (2pt); 
        \filldraw  (8,0) circle (2pt); 
        \filldraw  (7,0) circle (2pt); 
        \filldraw  (3,0) circle (2pt); 
        \filldraw  (5.5,0) circle (2pt); 
        \filldraw  (4,0) circle (2pt);

        \draw (-0.5,0) arc (180:0:3) ;
        \draw[dotted][thick] (3,0) arc (180:0:2.5) ;
        \draw[dotted][thick] (4,0) arc (180:0:1.5) ;
        % \draw (0,0) arc (0:180:4) ;
        % \draw (0,0) arc (-0:0:2) ;
        \node[] at (-0.5,-0.25)   { \tiny$i$};
       
        \node[] at (2.8,-0.3)   {\tiny {$-4$}};
        \node[] at (3.8,-0.3)   {\tiny {$-3$}};
        \node[] at (5.3,-0.3)   {\tiny {$-2$}};
         \node[] at (6.8,-0.3)   {\tiny $-1$};
         \node[] at (8,-0.25)   {\tiny $2n$};
        \node[] at (4,-1.25)   { $M_3$};
       
\end{tikzpicture}
\end{minipage}
\begin{minipage}{.25\textwidth}
     \begin{tikzpicture}[scale=.45]

      \filldraw  (0,0) circle (2pt); 
        \filldraw  (6,0) circle (2pt); 
        \filldraw  (5,0) circle (2pt); 
        \filldraw  (3,0) circle (2pt); 
        \draw (0,0) arc (180:0:3) ;
        \draw[dotted][thick] (3,0) arc (180:0:1) ;
        % \draw (0,0) arc (0:180:4) ;
        % \draw (0,0) arc (-0:0:2) ;
        \node[] at (0,-0.3)   { \tiny$i$};
        \node[] at (2.8,-0.3)   { \tiny$-2$};
         \node[] at (4.8,-0.3)   { \tiny$-1$};
         \node[] at (6,-0.3)   { \tiny$2n$};
         \node[] at (3,-1.25)   { $M_4$};
        
     \end{tikzpicture}
\end{minipage}

\begin{minipage}{.28\textwidth}
     \begin{tikzpicture}[scale=.55]
      \filldraw  (0,0) circle (2pt); 
        \filldraw  (1,0) circle (2pt); 
        \filldraw  (7,0) circle (2pt); 
        \filldraw  (6,0) circle (2pt); 
        \filldraw  (5,0) circle (2pt); 
        \filldraw  (4,0) circle (2pt);

        \draw (0,0) arc (180:0:3) ;
        \draw (1,0) arc (180:0:2) ;
        \draw[dotted][thick] (4,0) arc (180:0:1.5) ;
        % \draw (0,0) arc (0:180:4) ;
        % \draw (0,0) arc (-0:0:2) ;
        \node[] at (0,-0.3)   {  \tiny $i$};
        \node[] at (1,-0.3)   {  \tiny $j$};
        \node[] at (3.8,-0.3)   { \tiny {$-3$}};
        \node[] at (4.8,-0.3)   {  \tiny{$-2$}};
         \node[] at (5.8,-0.3)   {\tiny $-1$};
         \node[] at (7,-0.3)   { \tiny $2n$};
         \node[] at (3.5,-1.25)   { $M_5$};
     \end{tikzpicture}
\end{minipage}
\begin{minipage}{.34\textwidth}
     \begin{tikzpicture}[scale=.55]
 \filldraw  (-0.5,0) circle (2pt);
        \filldraw  (.5,0) circle (2pt);
        \filldraw  (8.25,0) circle (2pt); 
        \filldraw  (7.25,0) circle (2pt); 
        \filldraw  (2.75,0) circle (2pt); 
        \filldraw  (5.5,0) circle (2pt); 
        \filldraw  (4.5,0) circle (2pt); 
        \filldraw  (3.75,0) circle (2pt);

        \draw (-0.5,0) arc (180:0:3) ;
        \draw (.5,0) arc (180:0:2) ;
        \draw[dotted][thick] (2.75,0) arc (180:0:2.75) ;
        \draw[dotted][thick] (3.75,0) arc (180:0:1.75) ;
        % \draw (0,0) arc (0:180:4) ;
        % \draw (0,0) arc (-0:0:2) ;
        
        \node[] at (-0.5,-0.3)   { \tiny$i$};
         \node[] at (0.5,-0.3)   { \tiny$j$};
        \node[] at (2.55,-0.3)   {\tiny {$-5$}};       
        \node[] at (3.55,-0.3)   {\tiny {$-4$}};
        \node[] at (4.35,-0.3)   {\tiny {$-3$}};
        \node[] at (5.35,-0.3)   {\tiny {$-2$}};
         \node[] at (7.05,-0.3)   {\tiny $-1$};
         \node[] at (8.25,-0.3)   {\tiny $2n$};
     \node[] at (4,-1.25)   { $M_6$};
\end{tikzpicture}
\end{minipage}
\begin{minipage}{.3\textwidth}
     \begin{tikzpicture}[scale=.55]
     \filldraw  (-0.5,0) circle (2pt);
        \filldraw  (1.25,0) circle (2pt);
        \filldraw  (8.25,0) circle (2pt); 
        \filldraw  (6.5,0) circle (2pt); 
        \filldraw  (2.75,0) circle (2pt); 
        \filldraw  (5.5,0) circle (2pt); 
        \filldraw  (4.5,0) circle (2pt); 
        \filldraw  (3.75,0) circle (2pt);

        \draw (-0.5,0) arc (180:0:3) ;
        \draw (1.25,0) arc (180:0:1.25) ;
        \draw[dotted][thick] (2.75,0) arc (180:0:2.75) ;
        \draw[dotted][thick] (4.5,0) arc (180:0:1) ;
        % \draw (0,0) arc (0:180:4) ;
        % \draw (0,0) arc (-0:0:2) ;
        
        \node[] at (-0.5,-0.3)   { \tiny$i$};
         \node[] at (1.25,-0.3)   { \tiny$j$};
        \node[] at (2.55,-0.3)    {\tiny {$-5$}};       
        \node[] at (3.5,-0.3)   {\tiny {$-4$}};
        \node[] at (4.35,-0.3)   {\tiny {$-3$}};
        \node[] at (5.35,-0.3)   {\tiny {$-2$}};
         \node[] at (6.35,-0.3)   {\tiny $-1$};
         \node[] at (8.25,-0.3)   {\tiny $2n$};
     \node[] at (4,-1.25)   { $M_7$};
     \end{tikzpicture}
\end{minipage}
\caption{Permitted matchings}  
\label{Fig7}
 \end{figure}
Let $m_i(n)$ denote the number of matchings of type $M_i$ on $2n$ vertices,
for $i=1,2,\ldots,7$, and let $a_n$ be the total number of $P_3$-avoiding matchings
on $2n$ vertices.

Matchings of types $M_1$, $M_2$, and $M_4$ can be obtained by adding a single arc
to any $P_3$-avoiding matching on $2(n-1)$ vertices, without imposing any additional
conditions. Hence
$$m_1(n)=m_2(n)=m_4(n)=a_{n-1}.$$
Similarly, matchings of type $M_3$ are obtained by adding two arcs to any
$P_3$-avoiding matching on $2(n-2)$ vertices, and therefore $$m_3(n)=a_{n-2}.$$
To construct matchings of types $M_5$ and $M_6$, we must add a single arc to a
matching on $2(n-1)$ vertices that ends with two nested arcs. The matchings that
end with a single arc are precisely those of types $M_1$, $M_2$, and $M_5$,
and hence they cannot contribute to $M_5$. Consequently,
\begin{equation}\label{eq12}
    m_5(n)=a_{n-1}-m_1(n-1)-m_2(n-1)-m_5(n-1)
        =a_{n-1}-2a_{n-2}-m_5(n-1).
\end{equation}

Finally, matchings of types $M_6$ and $M_7$ are obtained by adding two arcs to a
matching on $2(n-2)$ vertices that ends with two nested arcs, and therefore
$$m_7(n)=m_6(n)=m_5(n-1).$$
Summing over all types, we obtain 
\begin{equation}\label{eq13}
    a_n=\sum_{i=1}^{7} m_i(n)=3a_{n-1}+a_{n-2}+m_5(n)+2m_5(n-1).
\end{equation}
Solving Equations \eqref{eq12} and \eqref{eq13} we get, 
\begin{align*}
    m_5(n-1)&=a_n-4a_{n-1}+a_{n-2}\\
    \mbox{Hence, }m_5(n) &=a_{n+1}-4a_{n}+a_{n-1}.
\end{align*}
Substitute the value of $m_5(n)$ and $m_5(n-1)$ in Equation \eqref{eq13} we get,
$$a_{n+1}=3a_{n}+4a_{n-1}-3a_{n-2},$$
which completes the proof.
\end{proof}

\begin{theorem}\label{linear_quintuple}
For $Q_3=\{123,132,213,312,321\}$ and $Q_5=\{123,213,231,312,321\}$, and let $a_n$
denote the number of matchings on $2n$ vertices avoiding $P$.
Then for all $n\ge 3$,
$$
a_0=1,\qquad a_1=1,\qquad a_2=3,
$$
and
$$
a_n=
\begin{cases}
3a_{n-1}+a_{n-2}, & \text{if } P=Q_3,\\[6pt]
4a_{n-1}-2a_{n-2}, & \text{if } P=Q_5.
\end{cases}
$$
\end{theorem}

\begin{proof}
Since $\{123,213,312,321\}\subset Q_3,Q_5$, the proof follows the same construction as in
Theorem \ref{linear_quadruple1}. In both cases, a new arc can be added only in certain
admissible configurations shown in Figure \ref{Fig2}. 
%Let $m_i(n)$ denote the number
%of matchings of type $M_i$ on $2n$ vertices.

In each case, configurations of types $M_1$, $M_2$, and $M_3$ are obtained by adding a
single arc to any $P$-avoiding matching on $2(n-1)$ vertices. Hence $m_1(n)=m_2(n)=m_3(n)=a_{n-1}.$

If $P=Q_3$, the only additional admissible configuration is $M_5$, which is obtained by
adding two arcs to a matching on $2(n-2)$ vertices. Thus $m_5(n)=m_3(n-1)=a_{n-2},$
and therefore $$ a_n=m_1(n)+m_2(n)+m_3(n)+m_5(n)=3a_{n-1}+a_{n-2}.$$

If $P=Q_5$, the additional admissible configuration is $M_4$, obtained by adding a
single arc to a matching on $2(n-1)$ vertices that does not end with a single arc.
Hence $$m_4(n)=a_{n-1}-m_1(n-1)-m_2(n-1)=a_{n-1}-2a_{n-2},$$
and consequently $$a_n=m_1(n)+m_2(n)+m_3(n)+m_4(n)=4a_{n-1}-2a_{n-2}.$$
This completes the proof.
\end{proof}

\begin{theorem}\label{2nCn}
For $P_7=\{132,213,231,312\},$ the number of connected matchings on $2n$ vertices avoiding $P_7$ is $\binom{2(n-1)}{n-1},$
for $n\geq 1$.
\end{theorem}
\begin{proof}
Recall that $\binom{2m}{m}$ counts the lattice paths from $(0,0)$ to $(2m,0)$ consisting of up-steps
$U=(1,1)$ and down-steps $D=(1,-1)$, with no restriction on going below the
$x$-axis. Such paths are sometimes called \emph{grand Dyck paths}. Let
$\mathcal{L}_{n-1}$ denote the set of all such paths from $(0,0)$ to
$(2(n-1),0)$. We construct a bijection
$$
\phi:\mathcal{M}_n^c(P_7)\longrightarrow\mathcal{L}_{n-1},
$$
where $\mathcal{M}_n^c(P_7)$ denotes the set of connected $P_7$-avoiding matchings on $[2n]$.
We first describe the structure of a $P_7$-avoiding matching. Since $
P_7=\{132,213,231,312\},$
the only allowed three-arc permutation patterns are $123$ and $321$,
corresponding respectively to three mutually nested arcs and three mutually
crossing arcs. Consequently, an arc cannot simultaneously be involved in a
crossing and be nested inside another arc, since together with the relevant
two arcs this would create one of the forbidden patterns in $P_7$.
Moreover, by connectedness, every arc other than a possible outer arc
joining the first and last vertices is either involved in a crossing or is
nested inside another arc. We call these \emph{crossing arcs} and
\emph{nested arcs}, respectively. Let $
M\in\mathcal{M}_n^c(P_7).$
Ignore the first vertex and the last vertex of $M$, and label the remaining
vertices from left to right by $1,2,\ldots,2(n-1).$
We associate a step with each of these vertices according to the following
rules.
\begin{itemize}
\item If the $i$-th vertex is the opening vertex of a crossing arc, set
the $i$-th step equal to $U$.
\item If the $i$-th vertex is the closing vertex of a crossing arc, set
the $i$-th step equal to $D$.
\item If the $i$-th vertex is the opening vertex of an arc nested inside
another arc, set the $i$-th step equal to $D$.
\item If the $i$-th vertex is the closing vertex of an arc nested inside
another arc, set the $i$-th step equal to $U$.
\end{itemize}
These cases are mutually exclusive by the observation above.
If the first and last vertices belong to the same arc, then this outer arc is
ignored completely, and every remaining arc contributes exactly one up-step
and one down-step. If the first and last vertices belong to distinct arcs,
then connectedness and $P_7$-avoidance imply that both of these arcs are
crossing arcs. The arc containing the first vertex contributes only its
closing vertex, hence one down-step, while the arc containing the last vertex
contributes only its opening vertex, hence one up-step. All other arcs again
contribute one up-step and one down-step.

Thus, in every case, the resulting path contains exactly $n-1$ up-steps and
$n-1$ down-steps. Hence it is a lattice path from $(0,0)$ to $(2(n-1),0),$
and therefore belongs to $\mathcal{L}_{n-1}$.

We now describe the inverse construction. Let $L\in\mathcal{L}_{n-1}.$
Decompose $L$ into maximal excursions between consecutive visits to the
$x$-axis. Each excursion lies either strictly above or strictly below the
$x$-axis except at its endpoints.
Begin with an unmatched opening vertex, which will be the first vertex of the
matching. We call it the \emph{distinguished opening vertex}. We read the excursions of $L$ from left to right.

Suppose first that an excursion lies above the $x$-axis. Maintain a queue
whose first element is the distinguished opening vertex. Reading the
excursion from left to right,
\begin{itemize}
\item for each up-step $U$, create a new opening vertex and append it to
the queue;
\item for each down-step $D$, create a closing vertex and join it to the
opening vertex at the front of the queue, then remove that opening vertex
from the queue.
\end{itemize}
Since the excursion contains equally many up- and down-steps, exactly one
opening vertex remains in the queue at the end. We take this vertex to be the
new distinguished opening vertex. The arcs created in this way form
crossings.
Now suppose that an excursion lies below the $x$-axis. Reading it from left
to right, use a stack:
\begin{itemize}
\item for each down-step $D$, create a new opening vertex and place it on
the top of the stack;
\item for each up-step $U$, create a closing vertex, join it to the
opening vertex on the top of the stack, and remove that opening vertex
from the stack.
\end{itemize}
The stack is empty at the end of the excursion. The arcs created in this way
are nested inside the distinguished arc.
After all excursions have been processed, one distinguished opening vertex
remains unmatched. Add the final vertex and join it to this distinguished
opening vertex.
The resulting matching is connected. Moreover, the queue construction above
the $x$-axis produces crossing configurations, while the stack construction below the $x$-axis produces nested configurations. The two constructions are joined only through the distinguished arc. Hence any three arcs that form a permutation pattern are either mutually nested or mutually crossing, giving
the patterns $123$ or $321$, respectively. Therefore none of
$132, 213, 231, 312$ occurs, and the resulting matching avoids $P_7$.

The two constructions are inverse to each other. Indeed, the vertices
belonging to crossing arcs correspond exactly to the steps occurring above
the $x$-axis, while those belonging to nested arcs correspond exactly to the
steps occurring below the $x$-axis. The queue and stack rules uniquely
recover the corresponding arcs.
Therefore,
$$
\mathcal{M}_n^c(P_7)\longleftrightarrow\mathcal{L}_{n-1}
$$
is a bijection. Consequently,
$$
|\mathcal{M}_n^c(P_7)|
=
|\mathcal{L}_{n-1}|
=
\binom{2(n-1)}{n-1}.
$$
We illustrate the bijection between  $\mathcal{M}^c_n(P_7)$ to $\mathcal{L}_{n-1}$ with an example (see Figure \ref{Figg3}).
\begin{figure}[H]
\centering
{
\begin{minipage}{.35\textwidth}

\begin{tikzpicture}[scale=.7]
	
        \filldraw  (0,0) circle (2pt); 
        \filldraw  (1,0) circle (2pt); 
        \filldraw  (2,0) circle (2pt); 
        \filldraw  (6,0) circle (2pt); 
        \filldraw  (5,0) circle (2pt); 
        \filldraw  (3,0) circle (2pt); 
        \filldraw  (7,0) circle (2pt); 
        \filldraw  (8,0) circle (2pt);

        \draw (0,0) arc (180:0:1) ;
        \draw (1,0) arc (180:0:3) ;
        \draw (3,0) arc (180:0:1) ;
        \draw (6,0) arc (180:0:1) ;
        % \draw (0,0) arc (-0:0:2) ;
        \node[] at (1,-0.4)   {  \tiny $1$};
        \node[] at (2,-0.4)   {  \tiny $2$};
        \node[] at (3,-0.4)   {  \tiny $3$};
        \node[] at (5,-0.4)   { \tiny {$4$}};
        \node[] at (6,-0.4)   {  \tiny{$5$}};
         \node[] at(7,-0.4)   {\tiny $6$};
      
\end{tikzpicture}
\end{minipage}
\begin{minipage}{.08\textwidth}
   $\longleftrightarrow$
\end{minipage}
\begin{minipage}{.35\textwidth}
\begin{tikzpicture}[scale=.7]
\draw[->] (-0.5,0) -- (7,0) node[right] {$x$};
\draw[<->] (0,-2) -- (0,2) node[above] {$y$};
        \filldraw  (0,0) circle (2pt); 
        \filldraw  (0,0) circle (2pt); 
        \filldraw  (1,1) circle (2pt); 
        \filldraw  (2,0) circle (2pt); 
        \filldraw  (3,-1) circle (2pt); 
        \filldraw  (4,0) circle (2pt); 
        \filldraw  (5,1) circle (2pt); 
        \filldraw  (6,0) circle (2pt); 

    \draw[thick] (0,0) -- (1,1);
    \draw[thick] (1,1) -- (2,0);
    \draw[thick] (2,0) -- (3,-1);
    \draw[thick] (3,-1) -- (4,0);
    \draw[thick] (4,0) -- (5,1);
    \draw[thick] (6,0) -- (5,1);

     \node[] at (-0.7,-.5)   { $(0,0)$};
    \node[] at (0.3,.7)   { $1$};
    \node[] at (1.8,.7)   { $2$};
    \node[] at (2.3,-.7)   { $3$};
    \node[] at (3.7,-.7)   { $4$};
    \node[] at (4.3,.7)   { $5$};
    \node[] at (5.7,.7)   { $6$};
     \node[] at (6,-.5)   { $(6,0)$};
     %\draw[ gray!50] (1,1) -- (2,0);
\end{tikzpicture}
\end{minipage}
}
\caption{The correspondence between a $P_7$-avoiding connected matching and its balanced lattice path}
\label{Figg3}
 \end{figure}
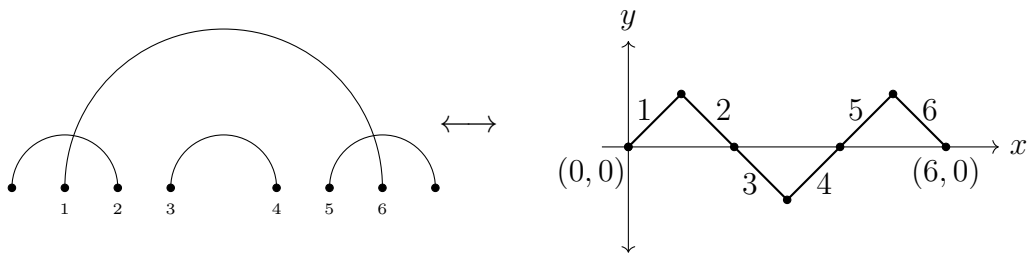
\end{proof}

\begin{theorem}\label{2nCn_rec}
    For $P_7=\{132,213,231,312\}$, then $a_n$ satisfies the following recursion
    $$a_n=\binom{2(n-1)}{n-1} + \sum_{k=1}^{n-1}a_k \binom{2(n-1-k)}{n-1-k} \mbox{ for }n\geq 2, ~a_0=a_1=1.$$
\end{theorem}
\begin{proof}
    We prove this by induction. Assume that for all $l \leq n-1$, the number of matchings on $2l$ vertices avoiding $P_7$ is $a_l$, and consider a matching on $2(n-1)$ vertices that avoids $P_7$. After suitably relabeling the vertices, we add a rightmost opening arc to obtain a matching on $2n$ vertices that still avoids $P_7$. Since there are $a_{n-1}$ such matchings on $2(n-1)$ vertices, we claim that this construction yields exactly $t_n$ matchings on $2n$ vertices.

By theorem \ref{2nCn}, the number of connected matchings on $2k$ vertices avoiding $P$ is $\binom{2(k-1)}{k-1}$. Let $M_k$ denote the rightmost connected component on $2k$ vertices for $k=0,1,\ldots,n-1$. If the newly added rightmost opening arc is attached to a component $M_k$, it forms a connected component $M_{k+1}$, while the remaining part of the matching remains unchanged. On the other hand, if the new arc is not connected to the existing structure, it forms a new rightmost connected component $M_1$ consisting of a single arc.

Therefore, the total number of matchings on $2n$ vertices avoiding $P_7$ is given by
$$\binom{2(n-1)}{n-1} + a_1 \binom{2(n-2)}{n-2} + a_2 \binom{2(n-3)}{n-3} + \cdots + a_{n-2}\binom{2}{1} + a_{n-1} = a_n.$$ This completes the proof.
\end{proof}

\begin{theorem}\label{large schroeder}
For $P_{10}=\{213,231,312,321\},$
the number of connected matchings on $2n$ vertices avoiding $P_{10}$ is $S_{n-1},$
where $S_m$ denotes the $m$-th large Schr\"oder number.
\end{theorem}

\begin{proof}
Recall that $S_{n-1}$ counts the Schr\"oder paths from $(0,0)$ to
$(2(n-1),0)$ with up-steps $U=(1,1)$, down-steps $D=(1,-1)$, and
horizontal steps $H=(2,0)$, which stay weakly above the $x$-axis.
Let $\mathcal{L}_{n-1}$ denote the set of all such paths.

We construct a bijection
$$
\phi:\mathcal{M}_n^c(P_{10})\longrightarrow \mathcal{L}_{n-1},
$$
where $\mathcal{M}_n^c(P_{10})$ denotes the set of connected
$P_{10}$-avoiding matchings on $[2n]$.

We first make a simple observation. Suppose two arcs
$$
A=(a,c)
\qquad\text{and}\qquad
B=(b,d)
$$
cross, so that $a<b<c<d.$
Then, in a $P_{10}$-avoiding matching, the two middle endpoints $b$ and
$c$ must be consecutive.
Indeed, suppose that there is a vertex $x$ with $b<x<c,$
and let $y$ be the vertex matched with $x$. If $x<y$, then according
to the position of $y$ relative to $c$ and $d$, the three arcs
$A,B,(x,y)$ form one of the patterns $213$, $312$, or $321$.
Similarly, if $y<x$, then according to the position of $y$ relative
to $a$ and $b$, the three arcs form one of the patterns $213$, $231$,
or $321$. In every case a pattern from $P_{10}$ occurs, a
contradiction. Hence $c=b+1.$

Now let $M\in\mathcal{M}_n^c(P_{10})$. The first vertex is necessarily
an opening vertex and the last vertex is necessarily a closing vertex.
Ignore these two vertices and read the remaining vertices from left to
right. Thus, the vertices $2,3,\ldots,2n-1$ correspond to the
$2(n-1)$ units of the path.
At each stage, apply the following rules.
\begin{itemize}
    \item If the current vertex is an opening vertex and the next
    vertex is a closing vertex belonging to a different arc, then
    these two vertices are the middle endpoints of a crossing pair.
    By the observation above, they are consecutive. Encode these two
    vertices together by a horizontal step $H=(2,0)$.
    \item If the current vertex is an opening vertex which is not the
    first vertex of such a crossing pair, encode it by an up-step
    $U=(1,1)$.
    \item If the current vertex is a closing vertex which is not the
    second vertex of such a crossing pair, encode it by a down-step
    $D=(1,-1)$.
\end{itemize}
Since the first opening vertex and the last closing vertex are omitted,
there are $n-1$ opening vertices and $n-1$ closing vertices to be
encoded. Each horizontal step uses one opening and one closing vertex.
Hence, if the resulting path has $h$ horizontal steps, then it has
$(n-1-h)$ up-steps and $(n-1-h)$ down-steps. Therefore its total horizontal
length is
$$
2h+2(n-1-h)=2(n-1).
$$

We now show that the resulting path never goes below the $x$-axis.
Since $M$ is connected, after every proper initial segment of its
vertices there is at least one arc joining that segment to the
remaining vertices. Thus, at every stage before the last vertex, the
number of opened arcs is strictly greater than the number of closed
arcs. After omitting the first opening vertex, this means that the
number of up-steps encountered is always at least the number of
down-steps encountered. Horizontal steps do not change the height.
Hence $\phi(M)$ stays weakly above the $x$-axis and ends at
$(2(n-1),0)$. Therefore,
$$
\phi(M)\in\mathcal{L}_{n-1}.
$$

Now we describe the inverse map. Let
$L\in\mathcal{L}_{n-1}.$
We reconstruct a matching using a stack of currently open arcs.
Begin with a single unmatched opening vertex; this will be the first
vertex of the matching. We then read the steps of $L$ from left to
right.
\begin{itemize}
    \item For an up-step $U$, create a new opening vertex and place it
    on top of the stack.
    \item For a down-step $D$, create a closing vertex, join it to the
    opening vertex currently on top of the stack, and remove that
    opening vertex from the stack.
    \item For a horizontal step $H$, create two consecutive vertices:
    first an opening vertex $u$ and then a closing vertex $v$. Let
    $a$ be the opening vertex currently on top of the stack. Join
    $a$ to $v$, remove $a$ from the stack, and place $u$ on top of the
    stack. Thus the two arcs form a crossing whose middle endpoints
    $u$ and $v$ are consecutive.
\end{itemize}
Because $L$ stays weakly above the $x$-axis, the stack is never empty
during this procedure. After all steps have been read, exactly one
opening vertex remains unmatched. Create the final vertex and join it
to this opening vertex.

The resulting matching is connected, since at every proper cut there
is at least one open arc crossing the cut. Moreover, the construction
allows only two types of three-arc patterns. If no horizontal step is
involved, the stack rule produces three nested arcs, corresponding to
the pattern $123$. If a horizontal step is involved, the corresponding
crossing pair can only lie inside an outer arc, giving the pattern
$132$. Hence none of $213, 231, 312,321$ 
can occur. Therefore the resulting matching avoids $P_{10}$.

The two constructions are inverse to each other. Indeed, every
crossing in a $P_{10}$-avoiding matching has consecutive middle
endpoints, and hence is recorded by a unique horizontal step; all
remaining opening and closing vertices are recorded by up- and
down-steps, respectively. Conversely, the stack construction uniquely
recovers these arcs from the path.
Thus
$$
\mathcal{M}_n^c(P_{10})
\longleftrightarrow
\mathcal{L}_{n-1}
$$
is a bijection. Consequently,
$$
|\mathcal{M}_n^c(P_{10})|=S_{n-1}.
$$
We illustrate the bijection between $\mathcal{M}^c_n(P_{10})$ to $\mathcal{L}_{n-1}$ with an example (see Figure \ref{Fig4}).

\begin{figure}[H]
\centering
{
\begin{minipage}{.35\textwidth}

\begin{tikzpicture}[scale=.6]
	
        \filldraw  (0,0) circle (2pt); 
        \filldraw  (1,0) circle (2pt); 
        \filldraw  (2,0) circle (2pt); 
        \filldraw  (6,0) circle (2pt); 
        \filldraw  (5,0) circle (2pt); 
        \filldraw  (4,0) circle (2pt); 
        \filldraw  (7,0) circle (2pt); 
        \filldraw  (8,0) circle (2pt);

        \draw (0,0) arc (180:0:4) ;
        \draw (1,0) arc (180:0:2.5) ;
        \draw (2,0) arc (180:0:1) ;
        \draw (5,0) arc (180:0:1) ;
        % \draw (0,0) arc (-0:0:2) ;
        \node[] at (1,-0.4)   {  \tiny $1$};
        \node[] at (2,-0.4)   {  \tiny $2$};
        \node[] at (4,-0.4)   {  \tiny $3$};
        \node[] at (5,-0.4)   { \tiny {$4$}};
        \node[] at (6,-0.4)   {  \tiny{$5$}};
         \node[] at(7,-0.4)   {\tiny $6$};
      
\end{tikzpicture}
\end{minipage}
\begin{minipage}{.08\textwidth}
   $\longleftrightarrow$
\end{minipage}
\begin{minipage}{.35\textwidth}
\begin{tikzpicture}[scale=.7]
\draw[->] (-0.5,0) -- (7,0) node[right] {$x$};
\draw[->] (0,-0.5) -- (0,3) node[above] {$y$};
        \filldraw  (0,0) circle (2pt); 
        \filldraw  (0,0) circle (2pt); 
        \filldraw  (1,1) circle (2pt); 
        \filldraw  (2,2) circle (2pt); 
        \filldraw  (3,1) circle (2pt); 
        \filldraw  (4,1) circle (2pt); 
        \filldraw  (5,1) circle (2pt); 
        \filldraw  (6,0) circle (2pt); 

    \draw[thick] (0,0) -- (1,1);
    \draw[thick] (2,2) -- (1,1);
    \draw[thick] (2,2) -- (3,1);
    \draw[thick] (3,1) -- (4,1);
    \draw[thick] (4,1) -- (5,1);
    \draw[thick] (6,0) -- (5,1);

    \node[] at (-0.7,-.5)   { $(0,0)$};
    \node[] at (0.4,.9)   { $1$};
    \node[] at (1.3,1.8)   { $2$};
    \node[] at (2.7,1.8)   { $3$};
    \node[] at (3.5,1.4)   { $4$};
    \node[] at (4.5,1.4)   { $5$};
    \node[] at (5.7,.7)   { $6$};
     \node[] at (6,-.5)   { $(6,0)$};

     %\draw[ gray!50] (1,1) -- (2,0);
\end{tikzpicture}
\end{minipage}
}
\caption{The proposed correspondence between a $P_{10}$-avoiding connected matching and a Schr\"oder path}
\label{Fig4}
 \end{figure}
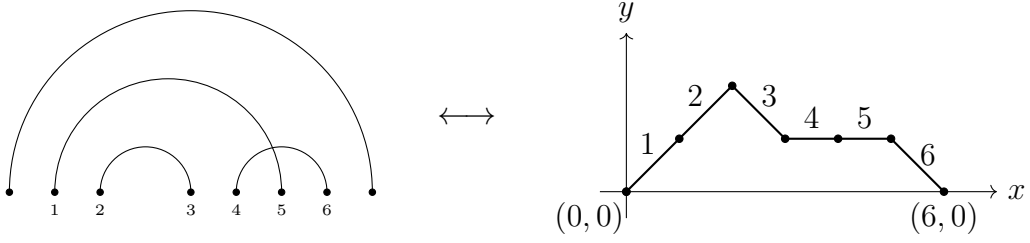
\end{proof}

\begin{theorem}\label{little schroeder}
    For $P_{10}=\{213,231,312,321\}$, then $a_{n}=t_n$ for all $n\geq 0$ where $t_n$ denotes the little Schr\"oder numbers.
\end{theorem}
\begin{proof}
  We prove the result by induction on $n$. For $n \geq 1$, recall that,
$$t_n = S_{n-1} + \sum_{k=1}^{n-1} t_k S_{n-1-k},$$
where $S_n$ denotes the large Schr\"oder numbers. The argument follows along the same lines as the proof of Theorem \ref{2nCn_rec}, and hence we omit the details.
% Assume that for all $l \leq n-1$, the number of matchings on $2l$ vertices avoiding $P$ is $t_l$, and consider a matching on $2(n-1)$ vertices that avoids $P$. After suitably relabeling the vertices, we add a rightmost opening arc to obtain a matching on $2n$ vertices that still avoids $P$. Since there are $t_{n-1}$ such matchings on $2(n-1)$ vertices, we claim that this construction yields exactly $t_n$ matchings on $2n$ vertices.

% From the previous theorem \ref{large schroeder}, the number of connected matchings on $2k$ vertices avoiding $P$ is $S_{k-1}$. Let $M_k$ denote the rightmost connected component on $2k$ vertices for $k=0,1,\ldots,n-1$. If the newly added rightmost opening arc is attached to a component $M_k$, it forms a connected component $M_{k+1}$, while the remaining part of the matching remains unchanged. On the other hand, if the new arc is not connected to the existing structure, it forms a new rightmost connected component $M_1$ consisting of a single arc.

% Therefore, the total number of matchings on $2n$ vertices avoiding $P$ is given by
% $$S_{n-1} + t_1 S_{n-2} + t_2 S_{n-3} + \cdots + t_{n-2} S_1 + t_{n-1} = t_n.$$ This completes the proof.
\end{proof}

\begin{lemma}\label{catalan}
For $Q_6=\{132,213,231,312,321\},$
the number of connected matchings on $2n$ vertices avoiding $Q_6$ is $C_n,$
where $C_n$ denotes the $n$-th Catalan number.
\end{lemma}

\begin{proof}
Recall that the Catalan number $C_n$ counts the Dyck paths of semilength $n$, that is, lattice paths from $(0,0)$ to $(2n,0)$ consisting of up-steps
$U=(1,1)$ and down-steps $D=(1,-1),$ 
which stay weakly above the $x$-axis. Let $\mathcal{D}_n$ denote the set of all such Dyck paths.
We describe a correspondence between $\mathcal{M}_n^c(Q_6)$, the set of connected $Q_6$-avoiding matchings on $[2n]$, and $\mathcal{D}_n$.
Since $Q_6=\{132,213,231,312,321\},$ 
the avoidance conditions strongly restrict the possible interactions among the arcs. In particular, the forbidden configurations shown in Figure \ref{Fig1} imply that the matching consists of nested arcs together with at most one crossing configuration.
Let $M\in\mathcal{M}_n^c(Q_6).$ 
Read the vertices of $M$ from left to right and associate a step with each vertex according to the following rules.
\begin{itemize}
    \item If the $i$-th vertex is an opening vertex of a noncrossing arc, or of an arc nested inside another arc, then associate with it an up-step $U$.

    \item If the $i$-th vertex is the corresponding closing vertex of such an arc, then associate with it a down-step $D$.

    \item For the vertices belonging to the crossing configuration, reverse these assignments: an opening vertex involved in the crossing contributes a down-step, while the corresponding closing vertex contributes an up-step.
\end{itemize}
Each arc contributes exactly one up-step and one down-step. Hence the resulting lattice path contains $n$ up-steps and $n$ down-steps, and therefore ends at
$(2n,0).$

We next verify that the path never goes below the $x$-axis. For the noncrossing and nested arcs, an opening vertex occurs before its corresponding closing vertex, and hence contributes an up-step before the associated down-step. The avoidance conditions on $Q_6$ restrict the crossing configuration so that reversing the two steps associated with a crossing does not cause the number of down-steps in any initial segment to exceed the number of up-steps. Therefore the resulting path remains weakly above the $x$-axis.
Thus the construction associates with every matching $M\in\mathcal{M}_n^c(Q_6)$
 a Dyck path of semilength $n$.

Conversely, the construction can be reversed. Reading a Dyck path from left to right, the up and down steps determine the opening and closing vertices of the nested arcs, while the unique local configuration corresponding to the reversed pair of steps determines the crossing configuration, when such a crossing is present. The avoidance conditions ensure that these arcs are paired uniquely and that no pattern in $Q_6$ is created.

Hence the correspondence is a bijection
$$
\mathcal{M}_n^c(Q_6)\longleftrightarrow\mathcal{D}_n.
$$
Therefore,
$$
|\mathcal{M}_n^c(Q_6)|
=
|\mathcal{D}_n|
=
C_n.
$$
An example illustrating this bijection is shown in Figure \ref{Fig5}.
\begin{figure}[H]
\centering
{
\begin{minipage}{.38\textwidth}

\begin{tikzpicture}[scale=.55]
	
        \filldraw  (0,0) circle (2pt); 
        \filldraw  (1,0) circle (2pt); 
        \filldraw  (2,0) circle (2pt); 
        \filldraw  (6,0) circle (2pt); 
        \filldraw  (5,0) circle (2pt); 
        \filldraw  (3,0) circle (2pt); 
        \filldraw  (9,0) circle (2pt); 
        \filldraw  (8,0) circle (2pt);

        \draw (0,0) arc (180:0:1) ;
        \draw (1,0) arc (180:0:4) ;
        \draw (3,0) arc (180:0:1) ;
        \draw (6,0) arc (180:0:1) ;
        % \draw (0,0) arc (-0:0:2) ;
        \node[] at (0,-0.4)   {  \tiny $1$};
        \node[] at (1,-0.4)   {  \tiny $2$};
        \node[] at (2,-0.4)   {  \tiny $3$};
        \node[] at (3,-0.4)   { \tiny {$4$}};
        \node[] at (5,-0.4)   {  \tiny{$5$}};
         \node[] at(6,-0.4)   {\tiny $6$};
          \node[] at(8,-0.4)   {\tiny $7$};
           \node[] at(9,-0.4)   {\tiny $8$};
      
\end{tikzpicture}
\end{minipage}
\begin{minipage}{.05\textwidth}
   $\longleftrightarrow$
\end{minipage}
\begin{minipage}{.35\textwidth}
\begin{tikzpicture}[scale=.6]
\draw[->] (-0.5,0) -- (9,0) node[right] {$x$};
\draw[->] (0,-0.5) -- (0,3) node[above] {$y$};
        \filldraw  (0,0) circle (2pt); 
        \filldraw  (0,0) circle (2pt); 
        \filldraw  (2,0) circle (2pt); 
        \filldraw  (1,1) circle (2pt); 
        \filldraw  (3,1) circle (2pt); 
        \filldraw  (4,2) circle (2pt); 
        \filldraw  (5,1) circle (2pt); 
        \filldraw  (6,2) circle (2pt); 
        \filldraw  (7,1) circle (2pt); 
        \filldraw  (8,0) circle (2pt); 

    \draw[thick] (0,0) -- (1,1);
    \draw[thick] (2,0) -- (1,1);
    \draw[thick] (2,0) -- (3,1);
    \draw[thick] (3,1) -- (4,2);
    \draw[thick] (4,2) -- (5,1);
    \draw[thick] (6,2) -- (5,1);
    \draw[thick] (6,2) -- (7,1);
    \draw[thick] (8,0) -- (7,1);

    \node[] at (-0.8,-.5)   { $(0,0)$};
    \node[] at (0.4,.9)   { $1$};
    \node[] at (1.6,.9)   { $2$};
    \node[] at (2.4,.9)   { $3$};
    \node[] at (3.2,1.8)   { $4$};
    \node[] at (4.7,1.8)   { $5$};
    \node[] at (5.3,1.8)   { $6$};
    \node[] at (6.6,1.8)   { $7$};
    \node[] at (7.8,.7)   { $8$};
     \node[] at (8,-.5)   { $(8,0)$};

     %\draw[ gray!50] (1,1) -- (2,0);
\end{tikzpicture}
\end{minipage}
}
\caption{The correspondence between a $Q_6$-avoiding connected matching and a Dyck path}
\label{Fig5}
 \end{figure}
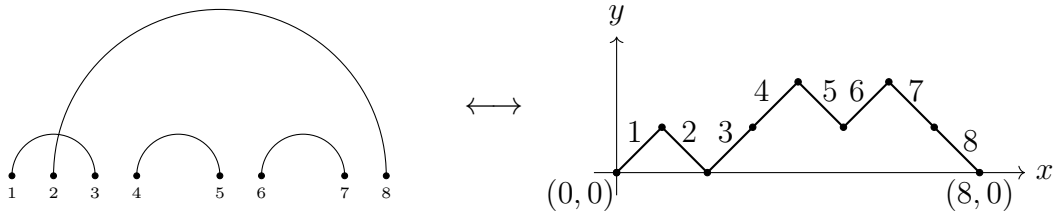
\end{proof}

\begin{theorem}\label{catalan2}
    Let $Q_6=\{132,213,231,312,321\}$, then $a_n$ satisfies the following recursion
    $$a_n=C_n + \sum_{k=1}^{n-1}a_k C_{n-k} \mbox{ for }n\geq 2, ~a_0=a_1=1.$$
\end{theorem}
\begin{proof}
    The result follows by induction on $n$, using an argument analogous to that in Theorem \ref{2nCn_rec}. We therefore omit the details.
\end{proof}

\section{Sets of Permutations Arising as Transversals} \label{sec-set}

In this section, we study the permutations that arise as $P$-avoiding transversals of Ferrers boards when $P$ is a quadruple or quintuple of patterns of length three. Building on the corresponding analysis for triples in Part I \cite{SUK-arxiv}, we determine the resulting ordinary permutation classes. Interestingly, these realizable permutations again form familiar classical avoidance classes. The following theorem catalogs the realizable permutations for these larger sets.

\begin{theorem}\label{set}
    The set of permutations that arise as transversals avoiding a given pattern set $P$ on some Ferrers board is precisely the set $S$. The sets $S$ and the corresponding pattern sets $P$ are described in Table \ref{set-table1} for quadruples and Table \ref{set-table2} for quintuples.

\begin{table}[H]
   
 \fontsize{9pt}{9pt}

 % \centering{
 
\begin{tabular}{|l|l|l|l|l|l|}
\noalign{\hrule height 1pt}
\textbf{No.} & ~~~~~~~~~~~~~~~~~~~~~~~~\textbf{Set of patterns} $\mathbf{P}$ & ~~~~~~~~~~~~~~~~$\mathbf{S}$ & \textbf{OEIS}\\
 \noalign{\hrule height 1pt}
1  & $\{132, 213, 231, 312\}$, $\{213, 231, 312, 321\}$, $\{132, 213, 312, 321\}$ & ~~~~~~~~~~~~$\SSS_n(213)$ &\href{https://oeis.org/search?q=A000108&go=Search}{A000108}\\

 \hline

 2 &~~~~~~~~~~~~~~~~~~~~~~~~~~ $\{123, 132, 231, 312\}$  & ~~~~~~~~~~~~$\SSS_n(123)$ &\href{https://oeis.org/search?q=A000108&go=Search}{A000108}\\
 \hline

3 & $\{123, 132, 213, 231\}$, $\{123, 132, 213, 312\}$, $\{123, 132, 213, 321\}$ & &\\

& $\{132, 213, 231, 321\}$, $\{123, 213, 231, 312\}$, $\{123, 213, 231, 321\}$&  ~~~~~~~~~$\SSS_n(123, 213)$ & \href{https://oeis.org/A000079}{A000079}\\

& ~~~~~~~~~~~~~~~~~~~~~~~~~~ $\{123, 213, 312, 321\}$ & &\\
 \hline

4& ~~~~~~~~~~~~~~~~~~~~~~~~~~ $\{132, 231, 312, 321\}$ & $\SSS_n(1324, 2314, 3124, 3214)$& \href{https://oeis.org/A000984}{A000984}\\
 \hline
5 & $\{123, 132, 231, 321\}$, $\{123, 132, 312, 321\}$, $\{123, 231, 312, 321\}$ &~~~~~~~~$\SSS_n(123, 3214)$& \href{https://oeis.org/ A001519}{A001519}\\

\noalign{\hrule height 1pt}

\end{tabular}
\caption{\small Quadruple pattern sets and the corresponding permutations arising as transversals of Ferrers boards  }
\label{set-table1}
\end{table} 

\begin{table}[H]
   
% \fontsize{9.5pt}{9.5pt}

  \centering{
 \small{
 {\renewcommand{\arraystretch}{1.5}
\begin{tabular}{|l|l|l|l|l|l|}
\noalign{\hrule height 1pt}
\textbf{No.} & ~~~~~~~~~~~~~~~~~~~~~~~~\textbf{Set of patterns} $\mathbf{P}$ & ~~~~~~~~~~~$\mathbf{S}$ &~~ \textbf{OEIS}\\
 \noalign{\hrule height 1pt}
1  & $\{123, 132, 213, 231, 312\}$, $\{123, 132, 213, 231, 321\}$ &~$\SSS_n(123,213)$ &\href{https://oeis.org/A000079}{A000079}\\
&$\{123, 132, 213, 312, 321\}$, $\{123, 213, 231, 312, 321\}$ & &\\
 \hline

 2 & ~~~~~~~~~~~~~~~~~~~~$\{132, 213, 231, 312, 321\}$  & ~~~~~$\SSS_n(213)$ &\href{https://oeis.org/search?q=A000108&go=Search}{A000108}\\
 \hline

3 & ~~~~~~~~~~~~~~~~~~~~$\{123, 132, 231, 312, 321\}$ &$\SSS_n(123, 3214)$ &\href{https://oeis.org/ A001519}{A001519}\\
\noalign{\hrule height 1pt}

\end{tabular}
}
}}
\caption{\small Quintuple pattern sets and the corresponding permutations arising as transversals of Ferrers boards  }
\label{set-table2}

\end{table} 
\end{theorem}

\begin{proof}
The structural arguments for these quadruples and quintuples are identical to those established in Theorem 30 in part I \cite{SUK-arxiv}. By applying the same iterative cell-deletion procedure to the bounding boxes of the respective patterns, we arrive at the corresponding permutation classes. Thus, we omit the explicit details.
\end{proof}

\bibliographystyle{acm}

\end{document}